\documentclass[a4paper,reqno,11pt]{amsart}
\usepackage{amssymb,amsfonts,amsmath,amsthm,graphicx}
 \usepackage[top=3cm, bottom=2.5cm, left=3cm, right=3cm, headsep=14pt]{geometry}
\numberwithin{equation}{section}

\newtheorem{thm}{Theorem}[section]
\newtheorem{corollary}[thm]{Corollary}
\newtheorem{lemma}[thm]{Lemma}
\newtheorem{prop}[thm]{Proposition}

\newtheorem{definition}[thm]{Definition}

\newtheorem{notation}[thm]{Notation}
\newtheorem{rmrk}[thm]{Remark}

\newcommand{\R}{\mathbb{R}}

\newcommand{\ba}{\begin{array}}
\newcommand{\ea}{\end{array}}

\newcommand{\bthm}{\begin{thm}}
\newcommand{\ethm}{\end{thm}}
\newcommand{\bstp}{\begin{stp}}
\newcommand{\estp}{\end{stp}}
\newcommand{\ephi}{\mathcal{E}_{\phi L}}
\newcommand{\ephij}{\mathcal{E}_{\phi_j L_j}}
\newcommand{\blemma}{\begin{lemma}}
\newcommand{\elemma}{\end{lemma}}
\newcommand{\bprop}{\begin{prop}}
\newcommand{\eprop}{\end{prop}}
\newcommand{\bpf}{\begin{pf}}
\newcommand{\epf}{\end{pf}}
\newcommand{\bdefn}{\begin{defn}}
\newcommand{\edefn}{\end{defn}}
\newcommand{\brk}{\begin{rmrk}}
\newcommand{\erk}{\end{rmrk}}
\newcommand{\bcrl}{\begin{crl}}
\newcommand{\ecrl}{\end{crl}}

\newcommand{\beqn}{\begin{equation}}
\newcommand{\eeqn}{\end{equation}}

\renewcommand{\leq}{\leqslant}
\renewcommand{\geq}{\geqslant}
\newcommand{\A}{\mathcal{A}}
\newcommand{\eps}{\varepsilon}

\newcommand{\delE}{\mathit{\Delta} E}
\def\Xint#1{\mathchoice
{\XXint\displaystyle\textstyle{#1}}%
{\XXint\textstyle\scriptstyle{#1}}%
{\XXint\scriptstyle\scriptscriptstyle{#1}}%
{\XXint\scriptscriptstyle\scriptscriptstyle{#1}}%
\!\int}
\def\XXint#1#2#3{{\setbox0=\hbox{$#1{#2#3}{\int}$}
\vcenter{\hbox{$#2#3$}}\kern-.5\wd0}}

\def\dashint{\Xint-}
\newcommand{\unif}{\bar{u}}
\newcommand{\sgn}{\text{sgn}}
\newcommand{\omegaint}{\int_{\Omega}}
\newcommand{\sigmad}{\sigma _d}
\newcommand{\reta}{r_\eta}
\newcommand{\esup}{\mathrm{ess}\,\mathrm{sup}}
\newcommand{\einf}{\mathrm{ess}\,\mathrm{inf}}

\newcommand{\omegphil}{\Omega_{\phi, L}}

\newcommand{\bal}{\begin{align}}
\newcommand{\elg}{\end{align}}
\newcommand{\baru}{\bar{u}}
\newcommand{\E}{\mathcal{E}}
\begin{document}
\title[Energy barrier in the Cahn-Hilliard equation]
{Energy barrier and $\Gamma$-convergence in the $d$-dimensional Cahn-Hilliard equation}
\vspace{1in}
\author{Michael Gelantalis}
\address{Michael Gelantalis, RWTH Aachen University}
\email{gelantalis@math1.rwth-aachen.de}
\author{Maria G. Westdickenberg}
\address{Maria G. Westdickenberg, RWTH Aachen University}
\email{maria@math1.rwth-aachen.de}
%\keywords{}
\begin{abstract}
We study the d-dimensional Cahn-Hilliard equation on the flat torus in a
parameter regime in which the system size is large and the mean value is
close---but not too close---to -1. We are particularly interested
in a
quantitative description of the  energy landscape in the case in which the uniform state is a
local but not global energy minimizer. In this setting, we derive a
sharp leading order estimate of the size of the energy barrier surrounding
the uniform state. A sharp interface version of the proof leads to a  $\Gamma$-limit of the rescaled energy gap between a given function and the uniform state.
\end{abstract}
%\dedicatory{}
\date{}
\maketitle
\section{Introduction}
We derive quantitative estimates on the energy barrier surrounding the uniform state in the d-dimensional Cahn-Hilliard equation on the torus in the metastable regime. These estimates are sharp at leading order and light the way to a $\Gamma$-limit for the rescaled energy gap. The study of the energy barrier is motivated by stochastics and the question of nucleation rates. It is well-known that a stochastic perturbation leads to so-called rare events or large deviations, in which the solution of a stochastically perturbed gradient flow ``hops'' from the basin of attraction of one local energy minimizer to that of another. The average timescale for such a rare event is exponentially large and the factor in the exponential is precisely one over noise strength times the energy barrier \cite{FW}. Hence, while we do not study a stochastic equation here, we derive analytical bounds on a deterministic quantity that has meaning for the related stochastic equation.

Because of its importance in nucleation phenomena---for instance in metallurgy, chemistry, and microelectronics---energy barriers and the corresponding ``critical nucleus'' have attracted widespread attention in various application areas ever since the pioneering work of Cahn and Hilliard \cite{CH,CH3}. For recent experimental and numerical studies of nucleation rates, see for instance \cite{LBM,LZZ,PB,ZLZ} and the many references therein.

In contrast, within the mathematical community around the calculus of variations, although the existence of energy barriers is exploited in the rich literature around mountain pass theorems, quantitative studies of energy barriers seem to be rare. Here we analyze an energy barrier and the corresponding $\Gamma$-limit in the context of the Cahn-Hilliard model
\begin{align}
u_t=-\Delta(\Delta u-G'(u))\label{chdyn}
\end{align}
for the mixing of a binary alloy \cite{CH}, where the order parameter $u$ indicates the percentage of material in each phase. From the mathematical point of view, a subtlety of the analysis is that we will consider the competing limits of \emph{large system size} and \emph{mean value close to -1}; see subsections \ref{ss:set} and \ref{ss:lit} below for details about this joint limit. In addition, the $\Gamma$-limit of the rescaled energy barrier represents a (simple) second order $\Gamma$-expansion of the energy; $\Gamma$-expansions have recently been explored by Braides and Truskinovsky \cite{BT}.

Fundamental for our work is the fact that equation~\eqref{chdyn} represents the $\dot{H}^{-1}$ gradient flow with respect to the energy
\begin{align}\label{Energy}
E(u):=\int _{\Omega}\frac{1}{2}|\nabla u|^2+G(u)\,dx.
\end{align}
The first term in the energy models an energetic penalization for spatial variations in $u$, while the second term---the so-called potential term---is a double well potential representing an energetic preference for the two pure phases. For simplicity, we consider the canonical double-well potential
$$G(u)=\frac{1}{4}(1-u^2)^2.$$
An important feature of the dynamic equation~\eqref{chdyn} is that it \emph{preserves the mean} of the order parameter.
Hence, one is interested in the properties of the energy considered for functions with fixed mean.

In the first part of our work, we analyze the energy barrier around the uniform state, that is, the difference between the energy of the minimum energy state on the boundary of the basin of attraction of the uniform state and the energy of the uniform state itself. In the second part of our work, using a sharp interface version of the preceding arguments, we derive the $\Gamma$-limit of the rescaled energy gap between a given function and the uniform state. The limit functional depends linearly on the perimeter and quadratically on the volume of the $+1$ phase in the limit. The limiting functional is predicted by the heuristics; see subsection \ref{ss:lit}.
\subsection{The energy barrier}\label{ss:set}
Consider dimension $d\geq 2$ and let $\Omega\subset \R^d$ be the flat $d$-dimensional torus of volume $L^d$, i.e., $\Omega=[-L/2,L/2]^d$ with periodic boundary conditions. We consider the energy over functions in $H^1\cap L^4(\Omega)$ with fixed mean $-1+\phi$, i.e.,
\begin{align}\label{H-phi}
X_\phi(\Omega):=\Big\{u\in H^1\cap L^4(\Omega): \dashint_\Omega u \, dx=-1+\phi\Big\}.
\end{align}
We are interested in the so-called off-critical parameter regime
\begin{align}\label{subcritical}
L\gg 1\qquad\text{and}\qquad L^{-d/(d+1)}\ll \phi \ll 1,
\end{align}
and the critical regime
\begin{align}\label{critical}
L\gg 1\qquad\text{and}\qquad L^{-d/(d+1)}\sim \phi \ll 1.
\end{align}
It is easy to see that the \emph{uniform state}
\begin{align*}
\bar{u}:= -1+\phi
\end{align*}
satisfies the mean constraint and is a local energy minimizer. In the off-critical regime, it is also easy to see  (cf. subsection \ref{ss:lit}) that $\baru$ is \emph{not the global energy minimizer}. In the critical regime with
\begin{align*}
  \phi=\xi L^{-d/(d+1)}\qquad\text{for fixed }\xi\in (0,\infty),
\end{align*}
the situation is more subtle, but Bellettini, Gelli, Luckhaus and Novaga \cite{BGLN} (for an open set with Lipschitz boundary) and Carlen, Carvalho, Esposito, Lebowitz and Marra \cite{CCELM} (for the torus) establish that there exists a sharp constant at which the global minimizer changes from a spatially uniform state to a nonuniform ``droplet'' state (see subsection \ref{ss:lit} below for more).

In the setting in which there exist states of lower energy than $\baru$ (i.e., in the off-critical regime and critical regime with $\xi$ sufficiently large), we are interested in estimating the size of the associated energy barrier, which we define in the following way.
\begin{definition}
The energy barrier $\mathit{\Delta}E$ surrounding $\baru$ is
\begin{align}\label{ebd}
\mathit{\Delta}E:=\inf_{\gamma \in \mathcal{A}}\max_{t\in [0,1]} \Big(E(\gamma(t))-E(\bar{u})\Big),
\end{align}
where
\begin{align}
\A:=\Big\{\gamma \in C([0,1];X_\phi(\Omega)):\gamma(0)=\bar{u}, E(\gamma(1))<E(\baru)\Big\}.\label{admis}
\end{align}
We use the term energy gap (to the uniform state) of a given function $u$ to refer to the energy difference
$E(u)-E(\baru)$.
\end{definition}
In joint work\footnote{F. Otto, M. G. Reznikoff, unpublished notes, 2004.} with Otto, the argument for which is also included in \cite{Rez}, the second author established the \emph{scaling} of the energy barrier in the off-critical regime, i.e., that there exist positive constants $C_1,\,C_2\in\R$ such that
\begin{align*}
C_1 \,\phi^{-d+1}\leq \delE\leq C_2\, \phi^{-d+1}.
\end{align*}
In theorem \ref{theorem} below, we ``close the gap'' between $C_1$ and $C_2$. Before giving the details of our results in subsection \ref{ss:results}, we explain the heuristics and comment on connections with existing literature.
\subsection{Heuristics and connections with previous results}\label{ss:lit}
To explain the heuristics, we consider the so-called sharp-interface limit. Starting with the energy \eqref{Energy} on a domain of length-scale $L$  and rescaling space by $\eps=L^{-1}$, one obtains
the $\eps$-dependent energy
\begin{align*}
E_\eps(u)=\omegaint \frac{\eps}{2}\,|\nabla u|^2+\frac{G(u)}{\eps}\, dx,
\end{align*}
where $\Omega$ is now order one.
For $\eps\ll 1$, any function $u_\eps$ with bounded energy satisfies $u_\eps\approx \pm 1$ on most of the domain and its energy concentrates on transition regions  between the two phases.  In the sharp interface limit $\eps\downarrow 0$, $u_\eps$ converges almost everywhere to $\pm 1$ and the energy converges in the sense of $\Gamma$-limits.
In particular, according to the seminal result of Modica and Mortola and its extensions (see \cite{MM,M,S}), the energy $E_\eps$ acting on functions with a given fixed (i.e., independent of $\eps$) mean $m\in (-1,1)$ $\Gamma$-converges to
\begin{align}\label{perE}
c_0\text{ Per}\{x\in \Omega: u(x)=+1\}.
\end{align}
Here $\text{Per}(A)$ represents the perimeter of $A$ in $\Omega$ and $c_0$ denotes the constant
\begin{align}\label{c0}
c_0=\int_{-1}^{1}\sqrt{2G(s)}ds=\frac{2\sqrt{2}}{3},
\end{align}
which is the cost of a one-dimensional transition layer:
\begin{align}
c_0 \, := \, \inf \left\{ \int_{-\infty}^{\infty}\frac{\eps}{2}\,u_x^2+\frac{G(u)}{\eps}\, dx \colon \, u(\pm \infty) = \pm 1 \right\}.\notag
\end{align}
In the current paper, we are interested \emph{in the energy barrier} and in considering \emph{simultaneously $L\gg 1$ and mean $-1+\phi$ for $\phi\ll 1$}. To get started, we turn to the sharp-interface limit.
In particular, since the sharp-interface limit measures the leading order contribution to the energy for $\eps\ll1$---or, equivalently, for $L\gg 1$---we can use sharp-interface pictures to understand why $\baru$ is not the global minimizer in the off-critical regime.
If $u=\pm 1$ almost everywhere, then the mean constraint implies
\begin{align*}
V_+-(L^d-V_+)=(-1+\phi)L^d,
\end{align*}
where $V_+$ is the volume of the set where $u=+1$. Solving for $V_+$ gives
\begin{align}
V_+=\frac{\phi L^d}{2}\ll L^d.\label{vplus}
\end{align}
The minimizer of the perimeter functional under this constraint is a ball where $u=+1$ inside of a background where $u=-1$. We observe that the leading order energy of a smooth approximation of such a ``droplet function'' scales like
\begin{align}\notag
\phi ^{(d-1)/d}L^{d-1},
\end{align}
which is much less than
\begin{align}\label{E(unif)}
E(\unif)=\frac{L^d}{4}\left(1-(-1+\phi)^2\right)^2=\phi ^2 L^d-\phi ^3 L^d+\frac{\phi ^4}{4}L^d
\end{align}
in the off-critical regime \eqref{subcritical}.
Hence $\baru$ is indeed not the global minimizer.
(Finding states of lower energy than $\baru$ in the critical regime is more subtle, since then the energy of the droplet state and the energy of the uniform state are of the same order.)

The simple argument above verifies that, in the off-critical regime, there \emph{exist states of much lower energy than the uniform state}. However we have not argued that the $\pm 1$ function considered above approximates the energy minimizer, and indeed, it does not.
An idea developed in \cite{BCK1} and exploited in \cite{BCK1,BCK2} to analyze the two-dimensional Ising model and in
\cite{BGLN,CCELM} to analyze the global minimizer of the Cahn-Hilliard energy is the following: Rather than putting \emph{all of the ``excess mass''} into a droplet of $u=+1$ (incurring a large perimeter cost) and achieving a bulk field of $u=-1$ (incurring zero bulk energy), it may be better to form a ``partial droplet'' of $u=+1$ (reducing the perimeter cost) and allow for a nonzero bulk cost. We sketch the argument from \cite{BCK1}. Suppose that for $\eta\in [0,1]$ one puts a volume fraction of $\eta V_+$ into a ball of $u=+1$ and distributes $(1-\eta)V_+$ outside of this ball, leading to a value $u=-1+\alpha$ in the bulk (where $\alpha$ is determined by $\eta$ and the mean constraint).
The sum of the leading order surface energy and bulk energy is
\begin{align}
  c_0 \sigma_d\left(\frac{\eta V_+}{\sigma_d/d} \right)^{(d-1)/d}+\frac{4(1-\eta)^2 V_+^2}{L^d}.\label{ld}
\end{align}
Here and throughout, $\sigmad$ denotes the surface area of the $(d  -  1)$- unit sphere in $\mathbb{R}^d$. Heuristically, $\eta=0$ corresponds to the uniform state $\baru$ and $\eta=1$ corresponds to a ``full droplet'' of volume $V_+$.

Recall the definition \eqref{vplus} of $V_+$ and approximate $E(\baru)=\phi^2L^d+h.o.t.$
Then one can use \eqref{ld} to approximate the energy gap between an $\eta$-droplet $u_{\eta,drop}$ and the uniform state
in terms of the partial volume $v:=\eta V_+$ as
\begin{align}\label{sll}
E(u_{\eta,drop})-E(\baru)&=\bar{C}_1 v^{(d-1)/d}-\bar{C}_2\phi v+\frac{\bar{C}_3v^2}{L^d}+h.o.t.,\notag\\
&=  \phi^{-d+1}f(\nu)+h.o.t.,
\end{align}
where we have defined  the rescaled volume $\nu:=\phi^{d}v$ and introduced the constants
\begin{align}
\bar{C}_1:=c_0\sigmad^{1/d}d^{(d-1)/d},\qquad \bar{C}_2:=4,\qquad \bar{C}_3:=4,\label{barc1}
\end{align}
and the function
\begin{align}
f(\nu):=\bar{C}_1 \nu^{(d-1)/d}-4\nu+\frac{4\nu^2}{\phi^{d+1}L^d}.\label{f}
\end{align}
For the critical scaling $\phi=\xi L^{-d/(d+1)}$, \eqref{f} defines the function $f_\xi :\mathbb{R}_+\to \mathbb{R}$ as
\begin{align}\label{glambda}
f_\xi(\nu):=\bar{C}_1 \nu^{(d-1)/d}-4\nu+4\xi^{-(d+1)}\nu^2.
\end{align}
One can observe that there is a crossover at the value
\begin{align}\label{xid}
\xi_d:=c_0^{d/(d+1)}\sigmad^{1/(d+1)}\frac{(d+1)}{4^{d/(d+1)}d^{1/(d+1)}}
\end{align}
in the sense that $f_\xi(\nu)\geq 0$ if $\xi<\xi_d$
while $f_\xi$ has a strictly positive global minimizer if $\xi>\xi_d$. See
Figure \ref{figure1}.
\begin{figure}[h]
\centering%
\includegraphics[scale=0.75]{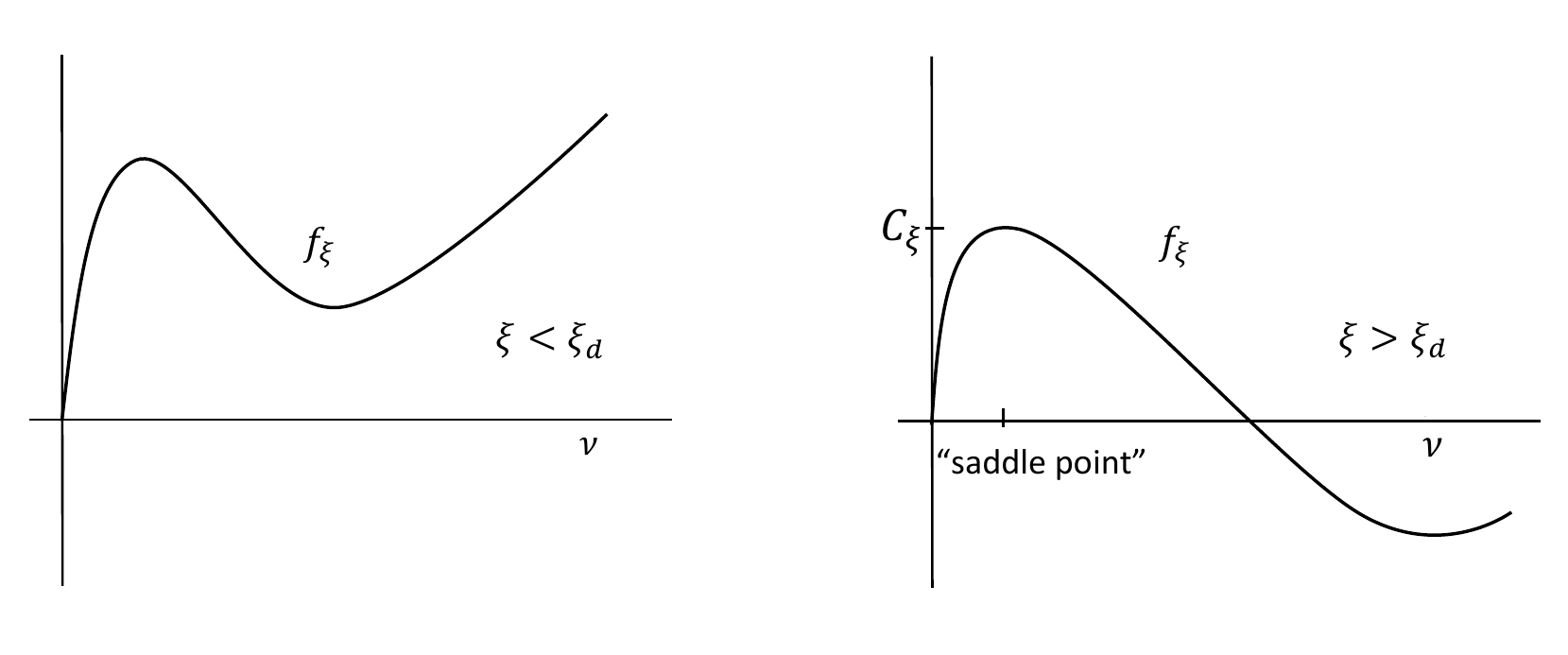}
\caption{The energy gap is ``approximated'' by $f_\xi$. The graph on the left corresponds to the case $\xi<\xi_d$, in which $f_\xi$ has a global minimum at $\nu =0$. The graph on the right corresponds to the case $\xi>\xi_d$, in which the global minimum occurs at a strictly positive value of $\nu$.}
\label{figure1}
\end{figure}

This heuristic analysis suggests that for $\xi<\xi_d$,
$\baru$ is the global energy minimizer, while for $\xi>\xi_d$, there exist states of lower energy. Exactly this fact is established in \cite{BGLN,CCELM}. (In \cite{BGLN}, see \cite [remark 2.5]{BGLN} and \cite [equation (2.9)]{BGLN}, which in our setting reduces to equation \eqref{xid}. In \cite{CCELM}, there is a typo in \cite [equation (1.21)]{CCELM}, but their argument leads indeed to the critical constant $\xi_d$ defined above in \eqref{xid}.)

Here we make additional use of the representation \eqref{glambda} to predict the size of the \emph{energy barrier}.
In the off-critical regime, one can argue that the third term in \eqref{sll} is higher order, so that the energy barrier is well approximated by
\begin{align}\label{f0}
f_\infty(\nu):=\bar{C}_1 \nu^{(d-1)/d}-4\nu.
\end{align}
Clearly $f_\infty$ takes on negative values, and one can check that $f_\infty$ attains its maximum at
\begin{align*}
\nu_m=\left(\frac{\bar{C}_1(d-1)}{4d}\right)^d,
\end{align*}
with maximum value
\begin{align}
  f_\infty(\nu_m)=\frac{\sigmad c_0^d}{d}\left(\frac{d-1}{4}\right)^{d-1}=:C_*.\label{constant}
\end{align}
Based on \eqref{sll} and \eqref{constant}, one may conjecture that the energy barrier in the off-critical regime is (to leading order) $C_*\phi^{-d+1}$. This is the content of \eqref{energybound} from theorem \ref{theorem}.
In the critical regime, all three terms in $f_\xi$ contribute and we cannot be as explicit about the barrier height. However for any $\xi>\xi_d$, let $\nu_\xi$ denote the first strictly positive zero of $f_\xi$ and define
\begin{align}\label{Cxi}
C_\xi:=\max_{\nu \in [0,\nu_\xi]}f_\xi (\nu).
\end{align}
The natural conjecture is that the energy barrier in the critical regime is (to leading order) $C_\xi\phi^{-d+1}$. This is the content of \eqref{energybound2} from theorem \ref{theorem}.

The fractional droplet functions considered in the above heuristics form the basis of the upper bound construction used in \cite{CCELM} to study global minimizers and below (cf. proposition \ref{upperbound}) to study the energy barrier.

Incidentally, in light of the heuristics explained above, one can observe in the off-critical regime the scale separation
\begin{align*}
 \delE\ll \inf_{X_\phi} E \ll E(\baru),
\end{align*}
while in the critical regime, one observes
\begin{align*}
 \delE\sim \inf_{X_\phi} E \sim E(\baru).
\end{align*}
On the one hand, the balance of terms in the critical regime makes certain calculations more delicate. On the other hand, in the off-critical regime, analyzing $\delE$ amounts to resolving a fine-scale feature of the energy landscape.
\subsubsection{Additional literature}
There is a vast literature on the Cahn-Hilliard equation, and in the preceding, we have only attempted to give a brief overview of the papers most closely related to our results and methods. We briefly summarize a few additional results that are related on some level to the present article.

In one space dimension and for large $L$ (or, equivalently, small $\eps$), a fundamental paper on the structure of minimizers is that of Carr, Gurtin, and Slemrod \cite{CGS} and a fundamental paper on the critical nucleus and nucleation is that of Bates and Fife \cite{BFi}. The structure of stable equilibria in higher dimensional problems is analyzed in \cite{GM}. Results on the sharp interface limit of stable equilibria are presented in \cite{SZ}, and the sharp interface limit of general critical points is analyzed in \cite{HT}.

There has been significant activity on the existence of so-called spike, bubble, and multi-spike solutions of the Cahn-Hilliard equation in \cite{WW,WW2,BDS,BFu} and related works; there have also been many recent results on spike solutions in similar models. While these works consider $L\gg 1$ (or $\eps\ll 1$) and fixed mean value, however, our interest is in the competition between small $\phi$ and large $L$. In particular, it is for $\phi\ll 1$ that the energy barrier becomes large and the saddle point acquires a sharp interface structure.

Our work also leads to questions about the structure and properties of the least energy saddle point (see corollary \ref{cor:sad} below). Related works include \cite{NT} and \cite{WW}.
\subsection{Results}\label{ss:results}
We now give the details of our results.
\begin{thm}[Energy barrier]\label{theorem}
Consider the Cahn-Hilliard energy functional \eqref{Energy} on $X_\phi$. In the off-critical regime \eqref{subcritical}, the energy barrier $\delE$ surrounding the uniform state is given by
\begin{align}\label{energybound}
\delE=C_* \phi ^{-d+1}+o(\phi ^{-d+1}),
\end{align}
where $C_*$ is defined in \eqref{constant}.

In the critical scaling
\begin{align*}
\phi=\xi L^{-d/(d+1)}
\end{align*}
for $\xi>\xi_d$  (where $\xi_d$ is defined in \eqref{xid}),
the uniform state $\bar{u}= -1+\phi$ is not the global energy minimizer and
the energy barrier $\delE$ surrounding $\baru$ is given by
\begin{align}\label{energybound2}
\delE=C_\xi \phi^{-d+1}+o(\phi ^{-d+1}),
\end{align}
where $C_\xi$ is defined in \eqref{Cxi}.
\end{thm}
The estimate \eqref{energybound} follows directly from the lower bound in proposition~\ref{prop:lowm} and the upper bound in proposition~\ref{upperbound}. The estimate \eqref{energybound2} follows directly from the lower bound in proposition~\ref{prop:lowmcrit} and the upper bound in proposition~\ref{upperbound}.
See also remark~\ref{rem:uplow} below.
\begin{rmrk}
  As remarked in subsection \ref{ss:lit}, the fact that $\baru$ is not the global energy minimizer in the critical regime with $\xi>\xi_d$ has already been established in \cite{BGLN,CCELM}.
\end{rmrk}
\begin{rmrk}
The energy barrier is  ``continuous''  with respect to the transition from the critical to the off-critical regime in the sense that, as $\xi \uparrow \infty$, $C_\xi \downarrow C_*$, where $C_\xi$ and $C_*$ are given by \eqref{Cxi} and \eqref{constant}, respectively.
\end{rmrk}
\begin{rmrk}[Relative size of the barrier] Although the energy barrier is  large, in the off-critical regime it is still \emph{much smaller than} the energy of the uniform state, since
\begin{align}\notag
\phi^{-d+1}\ll\phi ^2 L^d
\end{align}
in this case. In the critical regime, the energy barrier is of the same order as $E(\bar{u})$.
\end{rmrk}
\begin{rmrk}[Upper and lower bounds] \label{rem:uplow} As in \cite{Rez} and \cite{CCELM}, we will obtain \eqref{energybound} and \eqref{energybound2} with the method of upper and lower bounds. In proposition \ref{upperbound}, we construct a path connecting $\baru$ and a state of lower energy such that the maximum energy along the path is less than $C_* \phi^{-d+1}+o(\phi ^{-d+1})$ (in the off-critical regime) or $C_\xi \phi^{-d+1}+o(\phi ^{-d+1})$ (in the critical regime). In propositions \ref{prop:lowm} and \ref{prop:lowmcrit}, on the other hand, we establish that the maximum energy of any such continuous path is at least $C_* \phi^{-d+1}+o(\phi ^{-d+1})$ (in the off-critical regime) or $C_\xi \phi^{-d+1}+o(\phi ^{-d+1})$ (in the critical regime).

Our upper bound directly uses the idea of \cite{BCK1} and the construction of \cite{CCELM}. However whereas their interest was in the energy minimizer, we make the observation that the same construction can be used to build a good estimate of the energy barrier. See subsection~\ref{ss:lit} for an explanation of the idea of the construction and section \ref{s:up} for the construction itself.
\end{rmrk}
\begin{rmrk}[Same barrier for constrained Allen-Cahn]\label{rem:constrac} As remarked above, the energy barrier represents the exponential factor in the exponentially long timescale for the stochastically perturbed equation to leave the basin of attraction of the uniform state. Since the results in theorem~\ref{theorem} concern a static feature of the energy landscape---that is, since the energy barrier depends only on the energy and the mean constraint (and \emph{not on the metric})---they are the same for the Cahn-Hilliard equation and for the constrained Allen-Cahn equation
\begin{align*}
u_t=\Delta u-G'(u)+\frac{1}{L^d}\int_\Omega G'(u)\,dx.
\end{align*}
\end{rmrk}
A corollary of theorem \ref{theorem} is the existence of a saddle point, which follows from a standard mountain pass argument.
\begin{corollary}[Saddle point]\label{cor:sad}
For every pair $(\phi,L)$ related by $L^{-d/(d+1)}\ll \phi \ll 1$ or $\xi L^{-d/(d+1)}= \phi \ll 1$
for $\xi>\xi_d$,  the energy functional $E$ possesses  a nonconstant saddle point $u_s\in X_\phi$, which satisfies
\begin{align*}
-\Delta u_s+G'(u_s)=\lambda\qquad\text{ for some }\lambda\in\R
\end{align*}
and
\begin{align}
E(u_s)=E(\bar{u})+\mathit{\Delta} E.\label{ensad}
\end{align}
\end{corollary}
Although it seems possible that the existence of such a saddle point on the torus has been investigated before, we have not been able to find a reference in the literature. Therefore, for completeness, we include the mountain pass argument in the appendix. We emphasize, however, that our main contribution is not the existence of a saddle point, but rather the quantitative estimate of its energy given by inserting \eqref{E(unif)} and \eqref{energybound} (or \eqref{energybound2}) into \eqref{ensad}.

As a by-product of our study of the energy barrier, we obtain the $\Gamma$-limit of the rescaled energy gap
\begin{align}
\ephi(u)=\frac{E(u)-E(\bar{u})}{\phi^{-d+1}}.\label{rescaled1}
\end{align}
This second order $\Gamma$-limit is interesting because of the competing limits $\phi\downarrow 0$ and $L\uparrow\infty$ (see also remark \ref{rem:mm} below). Also we state the $\Gamma$-limit ``independent of boundary conditions'' in the sense that the $\Gamma$-limit is established with no assumption of periodicity; however see remark \ref{rem:perd} below.
\begin{thm}\label{Gamma}
Let
\begin{align*}
\Omega_{\phi, L}:=\left[-\frac{\phi L}{2},\frac{\phi L}{2}\right]^d,
\end{align*}
and for any $1<p<\infty$ consider the functional on $-1 + L^p(\mathbb{R}^d)$ defined by
\begin{align}\label{F_phi}
\ephi (u):=
\begin{cases}
\displaystyle \int_{\Omega_{\phi, L}}\dfrac{\phi}{2}|\nabla u|^2+\dfrac{1}{\phi}\Big(G(u)-G(-1+\phi)\Big)\,dx,\,\,\\ \qquad \mbox{if}\,\, u\in (H^1\cap L^4)(\Omega_{\phi, L}) \,\,\mbox{and}\,\,u=-1 \, \mbox{in} \,\,\mathbb{R}^d\setminus \Omega_{\phi, L},\\ \qquad \mbox{with}\,\,\dashint_{\Omega_{\phi, L}}  u \,dx=-1+\phi,\\ \\
+\infty,\quad\mbox{otherwise}.
\end{cases}
\end{align}
Let $0<\xi<\infty$ and suppose that for $\phi, L>0$ with $\phi \downarrow 0$ and $L\uparrow \infty$ there holds
\begin{align}
\phi L^{d/(d+1)}\to \xi.\label{lim}
\end{align}
Consider the functional on $-1 + L^p(\mathbb{R}^d)$ defined by
\begin{align}\label{F_0}
\mathcal{E}^\xi_0(u):=
\begin{cases}
c_0\text{Per}(C)-4|C|+4\xi^{-(d+1)}|C|^2,\,\,\\ \qquad \mbox{if}\,\, u=\pm 1 \,\mbox{a.e.}\,\,\mbox{and}\,\,\text{Per}(C)<\infty,\\
\qquad \mbox{where}\,\,C:=\{x: u(x)=+1\},\\ \\
+\infty,\quad\mbox{otherwise}.
\end{cases}
\end{align}
Then $\Gamma$-$\lim\ephi = \mathcal{E}^\xi_0$ in $-1 + L^p(\mathbb{R}^d)$ equipped with the topology generated by the usual $L^p(\mathbb{R}^d)$-distance.
\end{thm}
\begin{rmrk}\label{rem:mm}
Studying the behavior of $\ephi$ as $\phi \downarrow 0$ resembles the asymptotic problem of Modica and Mortola \cite{MM,M,S}; see subsection \ref{ss:lit}. However in our setting
\begin{align*}
  \phi L\sim\phi^{-1/d}\xi^{(d+1)/d},
\end{align*}
so that, rather than working with a fixed domain and fixed mean in $(-1,1)$, we consider
\begin{align*}
  \Omega_{\phi,L}\to\R^d\quad\text{and}\quad \dashint_{\Omega_{\phi,L}} u \, dx\to -1\quad\text{in the limit }\phi\downarrow 0.
\end{align*}
\end{rmrk}
\begin{rmrk}[Periodic boundary conditions]\label{rem:perd} From theorem \ref{Gamma} and its proof we also obtain $\Gamma$-convergence for the problem on the torus, i.e., if $\ephi$ is defined on $u\in (H^1\cap L^4)(\Omega_{\phi, L})$ subject to periodic boundary conditions and the mean constraint. The lower semicontinuity carries over automatically for sequences of periodic functions $u_\phi$ such that
\begin{align*}
  \int_{\Omega_{\phi,L}}|u_\phi-u_0|^p\,dx\to 0,
\end{align*}
and the recovery sequence that we define in step 2 of the proof is already periodic on $\Omega_{\phi, L}$.
\end{rmrk}
\begin{rmrk}
We remark that the topology of $L^1$ convergence overdetermines the problem in the following sense. Suppose that $u_\phi\to u_0$ in $L^1$ with $u_0=\pm 1$ a.e.\ and such that $C:=\{x\colon u_0(x)=1\}$ satisfies $|C|<\infty$. Then
\begin{align}
  \lim_{\phi\downarrow 0}\int_{\Omega_{\phi,L}}(u_\phi+1)\,dx=2|C|.\label{p1}
\end{align}
On the other hand, from the mean constraint, we have
\begin{align}
 \lim_{\phi\downarrow 0} \int_{\Omega_{\phi,L}}(u_\phi+1)\,dx= \lim_{\phi\downarrow 0} \phi|\Omega_{\phi,L}|\overset{\eqref{lim}}=\xi^{d+1}.\label{p2}
\end{align}
The combination of \eqref{p1} and \eqref{p2} implies that the limit defined in \eqref{lim} determines the measure of the limit set $C$ as $|C|=\xi^{d+1}/2$. To allow for different possible volumes of the set $C$, we consider $L^p$ convergence with $p>1$ (which, roughly speaking, allows the bulk value $-1+\alpha$ to converge to $-1$ at a slower rate than with $p=1$). Alternatively, in keeping with the heuristics explained in subsection \ref{ss:lit}, one can consider functions $u_\phi$ such that
\begin{align*}
  u_\phi-\alpha_\phi\chi_{\R^d\setminus C}\to u_0\qquad\text{in }L^1,
\end{align*}
where $\alpha_\phi:=\phi(1-2|C|/\xi^{d+1})$, $u_0=\pm 1$ a.e., and $C:=\{x\colon u_0(x)=1\}$.
\end{rmrk}
\begin{rmrk}[$\Gamma$-convergence in the off-critical regime]\label{rem:gamoff}
Using a slight modification of the proof of theorem \ref{Gamma}, one can establish $\Gamma$-convergence in the off-critical regime in the case that $\phi ^{d+1}L^d \uparrow \infty$ while $\phi ^{d+2}L^d \to 0$.
Let
\begin{align}
\mathcal{E}_0^\infty(u):=
\begin{cases}
c_0\text{Per}(C)-4|C|,\,\,\\ \qquad \mbox{if}\,\, u=\pm 1 \,\mbox{a.e.}\,\,\mbox{and}\,\,\text{Per}(C)<\infty,\\
\qquad \mbox{where}\,\,C:=\{x: u(x)=+1\},\\ \\
+\infty,\quad\mbox{otherwise}.
\end{cases}
\end{align}
Then for  every $2\leq p<\infty$, $\Gamma$-$\lim \ephi = \mathcal{E}_0^\infty$ in $-1 + L^p(\mathbb{R}^d)$ with regard to the $L^p(\mathbb{R}^d)$-topology.
\end{rmrk}

In this paper we derive the $\Gamma$-limit via a sharp interface version of the proof of the energy barrier. It would be interesting to consider things ``the other way around,'' i.e., to derive information about the energy gap $\E_\phi$ for $\phi>0$ from the (simpler) limit problem $\E^\xi_0$ or $\E_0^\infty$ in the critical or off-critical regimes, respectively. In addition, it would be interesting to analyze the structure of the saddle point $u_s$ from corollary \ref{cor:sad}. This is the subject of work in progress.
\subsection{Notation and organization}
\begin{notation}
If $X$ and $Y$ are nonnegative quantities, we write $X \lesssim Y$ to indicate that there exists a positive constant $C$ that depends at most on the dimension $d$ (except as indicated in the proof of theorem \ref{Gamma}), and such that $X \leq C Y$. Writing $X \sim Y$ means that $X \lesssim Y$ and $Y \lesssim X$.

We write $X\ll Y$ as $\phi\to 0$ if $X/Y\to 0$ as $\phi\to 0$. More generally, if $W, X, Y, Z$ are nonnegative quantities, then ``$W\ll X$ implies $Y\ll Z$'' means that for every $C_1<\infty$ there exists $C_2<\infty$ such that $W\leq X/C_2$ implies $Y\leq Z/C_1$.

In addition we use the standard $O(\cdot)$ and $o(\cdot)$ notation.
\end{notation}
\begin{notation}
If $A$ is a measurable subset of $\mathbb{R}^d$, we write $|A|$ to denote its $d$-dimensional Lebesgue measure. We use $\chi_A$  for the characteristic function of $A$.
\end{notation}
In section \ref{s:low} we establish the lower bounds for the off-critical and critical regimes (in propositions \ref{prop:lowm} and \ref{prop:lowmcrit}, respectively). In section \ref{s:up} we prove the upper bounds for the energy barrier (in proposition \ref{upperbound}). We derive the $\Gamma$-limit of the rescaled energy gap in section \ref{Gconv}. Finally, in the appendix we include the mountain pass argument for the existence of the saddle point.
\section{Lower bounds}\label{s:low}
Here we present the lower bounds for the energy barrier $\delE$ in the off-critical and critical regimes.
The first idea, exploited also in the scaling bound with Otto (see \cite{Rez}), is to  ``smuggle in the mean constraint'' by writing the energy gap of $u\in X_\phi$ in the form
\begin{align}
\lefteqn{E(u)-E(\bar{u})}\notag\\
&=E(u)-E(\bar{u})-G'(-1+\phi)\int_\Omega (u-(-1+\phi))\,dx\notag\\
&=\int_\Omega \frac{1}{2}|\nabla u|^2+G(u)-G(-1+\phi)-G'(-1+\phi)(u-(-1+\phi))\,dx\notag\\
&=:\int_\Omega e(u)\,dx.\label{smuggle}
\end{align}
The second idea, used also in \cite{Rez} and \cite{CCELM} in a somewhat different form, is to estimate separately the integral of $e(u)$ over the regions where $u$ is (roughly speaking) close to $+1$, close to $-1$, and strictly in between $\pm 1$.
To implement this idea, we introduce a partition of unity of $\mathbb{R}$ into three nonnegative smooth functions $\chi_1, \chi_2$, and $\chi_3 :\mathbb{R}\to [0,1]$, such that $\chi_1(t)+\chi_2(t)+\chi_3(t)=1$ for all $t\in\mathbb{R}$ and so that for some small $\kappa >0$ (we will fix $\kappa=\phi^{1/3}$ in the proof) we have
\begin{align}
\chi_1(t)&=\label{chi1}
\begin{cases}
1\quad&\mbox{for}\quad t\leq -1+\kappa\\
0\quad&\mbox{for}\quad t\geq -1+2\kappa,
\end{cases}\\
\chi_2(t)&=\label{chi2}
\begin{cases}
1\quad&\mbox{for}\quad -1+2\kappa\leq t\leq 1-2\kappa\\
0\quad&\mbox{for}\quad t\leq -1+\kappa \quad \mbox{and} \quad t\geq 1-\kappa,
\end{cases}\\
\chi_3(t)&=\label{chi3}
\begin{cases}
1\quad&\mbox{for}\quad t\geq 1-\kappa\\
0\quad&\mbox{for}\quad t\leq 1-2\kappa.
\end{cases}
\end{align}
See Figure \ref{figure2}.
\begin{figure}[h]
\centering%
\includegraphics[scale=0.80]{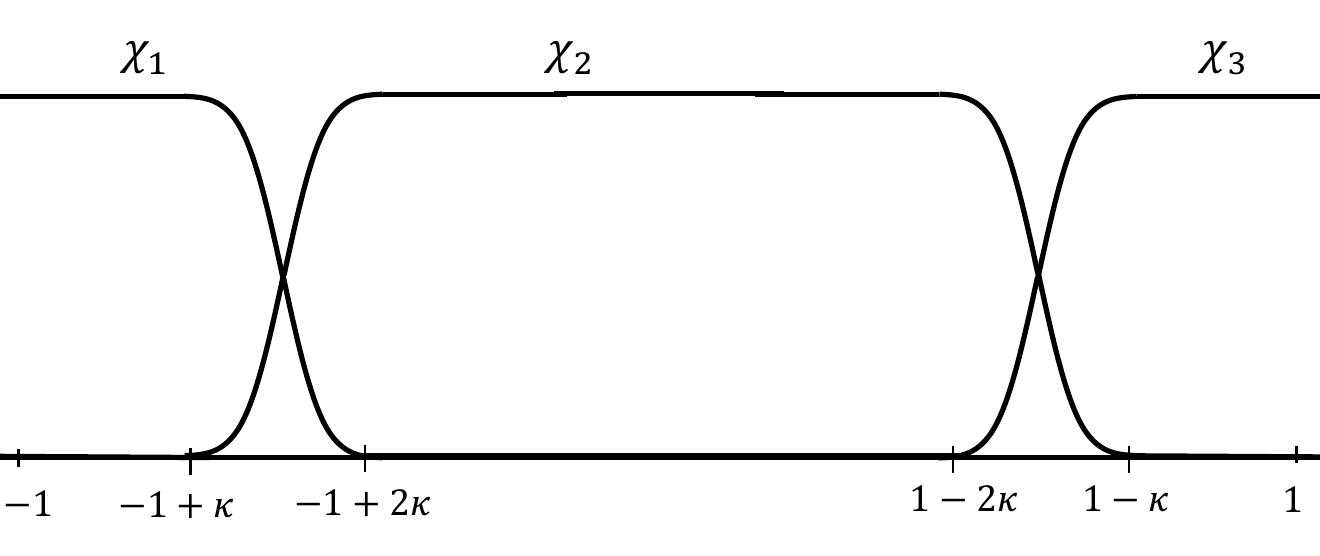}
\caption{The partition of unity $\chi_1+\chi_2+\chi_3=1$.}
\label{figure2}
\end{figure}

We use this partition of unity to decompose the energy gap as
\begin{align}
 \lefteqn{ E(u)-E(\baru)}\notag\\
 &=\int_\Omega e(u)\chi_1(u)\,dx+\int_\Omega e(u)\chi_2(u)\,dx+\int_\Omega e(u)\chi_3(u)\,dx.\label{deco}
\end{align}
Roughly speaking, we think of the support of $\chi_1(u(x))$ as the bulk phase, the support of $\chi_2(u(x))$ as the transition region(s), and the support of $\chi_3(u(x))$ as the ``droplet region'' where $u\approx +1$. However no assumption will be made about the geometry of the droplet region and we will not need to assume closeness to $-1$ in the bulk phase or to $+1$ in the droplet.

As a stand-in for the ``volume of the $+1$ phase,'' we define the continuous functional $V: X_\phi \to \mathbb{R}$ by
\begin{align}\label{V}
V(u):=\int_\Omega\chi_3(u)\,dx.
\end{align}
Our lower bound is given in terms of $V$. For clarity of exposition, we first consider the (simpler) off-critical regime in subsection \ref{ss:lower}. Then in subsection \ref{ss:lowcr}, we derive an improved lower bound for the critical regime.
\subsection{Lower bound in the off-critical regime}\label{ss:lower}
The idea for the lower bound is the following. First in \eqref{m:low}, we establish a lower bound on the \emph{energy gap} as a function of $V$ (defined above). Then we use \eqref{m:low} to establish a lower bound for the \emph{energy barrier}. (Notice that, according to proposition \ref{upperbound},  the energy barrier is well defined in the off-critical regime \eqref{subcritical}, so the lower bound \eqref{m:lowe} on $\delE$ is not vacuous.) The argument is elementary: Since the
function of $V$ on the right-hand side of \eqref{m:low} is zero at zero, is positive for $V$ small, and takes on a value $C_*\phi^{-d+1}+o(\phi^{-d+1})$ before reaching negative values, it is easy to see that any continuous path $\gamma\in \A$ (cf.\ \eqref{admis}) with $V(\gamma(0))=V(\baru)=0$ and $E(\gamma(1))-E(\baru)<0$ must have energy $C_*\phi^{-d+1}+o(\phi^{-d+1})$ for some $t\in (0,1)$.
\begin{prop}[Lower bound, off-critical regime]\label{prop:lowm}
For any $d\geq 2$, there exists $\epsilon_0(d)>0$ with the following property. In the off-critical scaling \eqref{subcritical} with $\phi\leq (\epsilon_0/4)^{3/4}$, any $u\in X_\phi$ with $V(u)\leq\epsilon_0 L^d$ satisfies
  \begin{align}
    E(u)-E(\baru)\geq C_1(\phi)V(u)^{(d-1)/d}-\phi C_2(\phi)V(u),\label{m:low}
\end{align}
where
\begin{align}
  C_1(\phi)&=(1-8\phi^{1/3} )^{1/2} (c_0-8\sqrt{2}\phi^{2/3} )\sigmad^{1/d} d^{(d-1)/d},\label{c1}\\
  C_2(\phi)&=\frac{1}{\phi}(2+\phi^{1/3} )G'(-1+\phi).\label{c2}
\end{align}
As a consequence, the energy barrier defined in \eqref{ebd} satisfies the lower bound
\begin{align}
 \delE&\geq \sup\{C_1(\phi)v^{(d-1)/d}-\phi C_2(\phi)v\colon v\lesssim \phi L^d\}\notag\\
 &= C_*\phi^{-d+1}+o(\phi^{-d+1}),\label{m:lowe}
\end{align}
where $C_*$ is given by~\eqref{constant}.
\end{prop}
\begin{rmrk}
  Notice that
  \begin{align*}
    \lim_{\phi\downarrow 0}C_1(\phi)=\bar{C}_1,\qquad  \lim_{\phi\downarrow 0}C_2(\phi)=\bar{C}_2,
  \end{align*}
  for $\bar{C}_1,\,\bar{C}_2$ defined in \eqref{barc1}.
\end{rmrk}
\begin{proof}
We obtain \eqref{m:lowe} directly from \eqref{m:low}. Indeed, consider any continuous path $\gamma\in\A$ (where $\A$ is defined in \eqref{admis}). Notice that $V$ is continuous on $X_\phi$ and that $V(\baru)=0$. Then \eqref{m:low} together with the properties of the function $v\mapsto C_1(\phi)v^{(d-1)/d}-\phi C_2(\phi)v$ and the calculation
\begin{align}
  \sup&\{C_1(\phi)v^{(d-1)/d}-\phi C_2(\phi)v\colon v\lesssim \phi L^d\}\notag\\
  &=
\frac{\sigmad (c_0-8\sqrt{2}\phi^{2/3})^d(1-8\phi^{1/3})^{d/2}}{d}\left(\frac{d-1}{4+2\phi^{1/3}}\right)^{d-1}\phi^{-d+1}\notag\\
&=C_*\phi^{-d+1}+o(\phi^{-d+1})\label{c_*}
\end{align}
imply
\begin{align*}
  \max_{t\in[0,1]}\left( E(\gamma(t))-E(\baru)\right)\geq C_*\phi^{-d+1}+o(\phi^{-d+1}).
\end{align*}
Hence, it suffices to establish \eqref{m:low}.
We remark that we may without loss of generality assume that
\begin{align}
E(u)\leq 2E(\baru),\label{enbu}
\end{align}
since otherwise \eqref{m:low} holds trivially.

\underline{Step 1}. We decompose the energy gap as in \eqref{deco} with $\kappa=\phi^{1/3}$. (This choice of $\kappa$ is motivated by \eqref{doc}, below.) Our first step is to show that the contribution on the ``bulk'' is positive. By convexity of $G$ on $(-\infty,-1+2\kappa)$, we have for $u$ within the support of $\chi_1$ that
\begin{align*}
  G(u)\geq G(-1+\phi)+G'(-1+\phi)(u-(-1+\phi)),
\end{align*}
so that indeed
\begin{align}
  \int_\Omega e(u)\chi_1(u)\,dx\geq 0.\label{st1}
\end{align}

\underline{Step 2}. We now estimate the (negative) contribution from the ``droplet.'' We will often make use of the fact that convexity near $-1$ gives
\begin{align}
0=G(-1)&\geq G(-1+\phi)+G'(-1+\phi)(-1-(-1+\phi))\notag\\
&=G(-1+\phi)-\phi G'(-1+\phi).\label{mconv1}
\end{align}
Using \eqref{mconv1}, we observe that
\begin{align*}
  \int_\Omega e(u)\chi_3(u)\,dx&\geq \int_\Omega \Big(-G(-1+\phi)-G'(-1+\phi)(u-(-1+\phi))\Big)\chi_3(u)\,dx\\
  &\geq \int_\Omega -G'(-1+\phi)(u+1)\chi_3(u)\,dx.
\end{align*}
Hence, if we can establish
\begin{align}
u\leq 1+\phi^{1/3},\label{usmall}
\end{align}
we will obtain
\begin{align}
  \int_\Omega e(u)\chi_3(u)\,dx\geq \int_\Omega -G'(-1+\phi)(2+\phi^{1/3})\chi_3(u)\,dx,\label{st2}
\end{align}
which is the desired lower bound on the support of $\chi_3$. The estimate \eqref{usmall} is justified by the following lemma (proved at the end of the subsection), which says that if \eqref{usmall} does not hold, we can replace $u$ by a function $\tilde{u}$ in a way that is compatible with our estimates and so that $\tilde{u}$ satisfies \eqref{usmall}. Hence a bound for functions less than or equal to $1+\kappa$ suffices.
\blemma\label{lemma5}
There exists $\kappa_0>0$ with the following property. For all $0<\phi \ll \kappa \leq \kappa _0$ and  $u\in X_\phi$, there exists a function $\tilde{u} \in X_\phi$ such that
\begin{itemize}
\item   [(i)] $\,\,V(\tilde{u})=V(u)$,
\item   [(ii)] $\,\,\tilde{u}=u$ on $\{x:-1+\phi\leq u(x)\leq 1+\kappa\}$,
\item   [(iii)] $\,\,\esup_\Omega \tilde{u}\leq 1+\kappa$ in $\Omega$,
\item   [(iv)] $\,\,E(\tilde{u})\leq E(u)$.
\end{itemize}
\elemma
\underline{Step 3}. Finally, we need to estimate the contribution to the energy gap over the ``transition region.'' Recalling \eqref{mconv1} and $G'(-1+\phi)>0$ and observing that $u+1\leq 2$ on the support of $\chi_2$, we obtain
\begin{align}
  \int_\Omega e(u)\chi_2(u)\,dx \geq \int_\Omega \Big(\frac{1}{2}|\nabla u|^2 +G(u)-2G'(-1+\phi)\Big)\chi_2(u)\,dx.\label{r1}
\end{align}
To begin, we would like to absorb the negative term. For this, we notice that
\begin{align}
&G'(-1+\phi)\leq 2\phi\notag\\
&G(u)\geq\frac{1}{2}\phi^{2/3}\;\;\text{on the support of }\chi_2,\label{doc}
\end{align}
from which it follows that
\begin{align*}
  8\phi^{1/3}G(u)-2G'(-1+\phi)\geq 0 \;\;\text{on the support of }\chi_2.
\end{align*}
Hence, letting $\tilde{G}(u)=(1-8\phi^{1/3})G(u),$ and invoking the  inequality $a^2+b^2\geq 2|a||b|$ and the coarea formula, we pass from \eqref{r1} to
\begin{align}
   \int_\Omega e(u)\chi_2(u)\,dx &\geq \int_\Omega \Big(\frac{1}{2}|\nabla u|^2 +\tilde{G}(u)\Big)\chi_2(u)\,dx\notag\\
   &\geq \int_\Omega \sqrt{2\tilde{G}(u)}\,\chi_2 (u) |\nabla u|\,dx\notag\\
   &= \int_{-1+\kappa}^{1-\kappa} \sqrt{2\tilde{G}(s)}\,\chi_2(s) \text{Per}_\Omega (\{u>s\})\,ds\notag\\
   &\geq \int_{-1+2\kappa}^{1-2\kappa}\sqrt{2\tilde{G}(s)}\,\text{Per}_\Omega (\{u>s\})\,ds\notag\\
   &\geq \int_{-1+2\kappa}^{1-2\kappa}\sqrt{2\tilde{G}(s)}P(|\{u>s\}|)\,ds.
   \label{per1}
\end{align}
Here we have used the notation $\text{Per}_\Omega$ for the perimeter in the torus, which we have bounded below by the so-called perimeter functional $P(v)$, i.e., the minimal perimeter in the torus of a subset with volume $v$. As in \cite{CCELM}, we will need two facts about the perimeter functional on the torus: First, $P$ is mononotically increasing for $0\leq v\leq L^d/2$. Second, according to the isoperimetric inequality on the torus \cite{MoJo},
there exists $\epsilon=\epsilon (d)>0$ such that
\begin{align}\label{iso}
P(v)=\sigmad ^{1/d} d^{(d-1)/d} v^{(d-1)/d}\quad \mbox{for}\quad v\leq \epsilon L^d.
\end{align}
(This gives the only restriction on $\epsilon_0$ in the statement of our proposition.)

To apply these facts, we need to check that $|\{u>1-2\kappa\}|\leq\epsilon L^d$ and $|\{u>s\}|\leq L^d/2$ for all $s\in [-1+2\kappa,1-2\kappa]$. We will show the stronger statement
\begin{align*}
  |\{u>-1+2\kappa\}|\leq \epsilon L^d.
\end{align*}
Indeed, using the assumed bound on $V$ and $G(s)\geq \kappa^2/2$ on $[-1+\kappa,1-\kappa]$, we observe that
\begin{align}
  |\{u>-1+2\kappa\}|&\leq V(u)+\int_\Omega \chi_{\{-1+\kappa\leq u\leq 1-\kappa\}}\,dx\notag\\
  &\leq \epsilon_0 L^d+\frac{2}{\kappa^2}\int_\Omega G(u)\,dx\notag\\
  &\leq \epsilon_0 L^d+\frac{2}{\kappa^2}E(u)\overset{\eqref{enbu}}\leq \epsilon_0 L^d +4\phi^{4/3}L^d\leq \epsilon L^d,\notag
\end{align}
for $\epsilon_0= \epsilon/2$ and $\phi\leq (\epsilon_0/4)^{3/4}$.

Hence we may use the monotonicity of $P$ and the isoperimetric inequality on the torus to deduce from  \eqref{per1} that
\begin{eqnarray}
  \lefteqn{\int_\Omega e(u)\chi_2(u)\,dx}\notag\\
   &\geq& P(|\{u>1-2\kappa\}|)\int_{-1+2\kappa}^{1-2\kappa}\sqrt{2\tilde{G}(s)}\,ds\notag\\
  &\geq &P(|\{u>1-2\kappa\}|)(1-8\phi^{1/3})^{1/2}(c_0-8\sqrt{2}\phi^{2/3})\notag\\
  &\overset{\eqref{iso}}= &(1-8\phi^{1/3})^{1/2}(c_0-8\sqrt{2}\phi^{2/3}) \sigmad ^{1/d} d^{(d-1)/d} |\{u> 1-2\kappa\}|^{(d-1)/d}\notag\\
 &\geq& (1-8\phi^{1/3})^{1/2}(c_0-8\sqrt{2}\phi^{2/3}) \sigmad ^{1/d} d^{(d-1)/d} V(u)^{(d-1)/d}.\label{stp3}
\end{eqnarray}

Combining \eqref{st1}, \eqref{st2}, and \eqref{stp3} establishes \eqref{m:low}.
\end{proof}
We conclude this subsection with the proof of lemma \ref{lemma5}.
\begin{proof}[Proof of lemma \ref{lemma5}.]
If $u$ already satisfies (iii), then there is nothing to prove. Hence we may assume $|D|>0$ where $D:=\{x: u(x)>1+{\kappa}\}$. Define
\begin{align*}
\Omega_-&:=\{x:u(x)\leq -1+\phi\},
\end{align*}
and, for $\lambda \in [0,1]$, define the function
\begin{align*}
\tilde{u}_\lambda(x):=
\begin{cases}
\min\{u(x),1+\kappa \}\quad&\mbox{for}\quad x\in \Omega\setminus\Omega_-\\
(1-\lambda)u(x)+\lambda (-1+\phi)\quad&\mbox{for}\quad x\in \Omega_-.
\end{cases}
\end{align*}
It is easy to see that $\omegaint \tilde{u}_0\,dx<(-1+\phi)L^d$ and $\omegaint \tilde{u}_1\,dx>(-1+\phi)L^d$. Therefore by the continuity of $\lambda \mapsto \omegaint \tilde{u}_\lambda dx$, there exists $\lambda_u\in(0,1)$ such that $\omegaint \tilde{u}_{\lambda_u}dx=(-1+\phi)L^d$. The function $\tilde{u}:=\tilde{u}_{\lambda_u}$ belongs to $X_\phi$ and satisfies properties (i)-(iii).
It remains to check whether (iv) holds. We first observe that, because of (ii), we have
\begin{align*}
E(\tilde{u})-E(u)&=\int_{D\cup \Omega_-}\frac{1}{2}|\nabla \tilde{u}|^2-\frac{1}{2}|\nabla u|^2+G(\tilde{u})-G(u)\,dx\\
&\leq \int_{D\cup \Omega_-}G(\tilde{u})-G(u)\,dx,
\end{align*}
where the second inequality follows since the gradient term of the energy of $\tilde{u}$ on $D\cup \Omega_-$ is smaller than the corresponding term of the energy of $u$. We thus have
\begin{align}\label{energy ineq}
E(\tilde{u})-E(u)\leq \int_{D} G(1+\kappa )-G(u)\,dx+\int_{\Omega_-}G(\tilde{u})-G(u)\,dx.
\end{align}
The convexity of $G$ on $[1+\kappa, \infty)$ implies that
$$G(1+\kappa )-G(u)\leq -G'(1+\kappa)\big(u-(1+\kappa)\big)$$
on $D$. On the other hand, since $u
\leq \tilde{u}\leq -1+\phi$ on $\Omega_-$, the convexity of $G$ on $(-\infty,-1+\phi)$ implies
\begin{align*}
G(\tilde{u})-G(u)\leq G'(\tilde{u})(\tilde{u}-u)\leq G'(-1+\phi)(\tilde{u}-u).
\end{align*}
Inserting these two inequalities into \eqref{energy ineq} yields
\begin{align}
\lefteqn{E(\tilde{u})-E(u)}\notag\\
&\leq G'(1+\kappa)\int _{D}1+\kappa-u\,dx+G'(-1+\phi)\int _{\Omega_-}\tilde{u}-u\,dx\notag\\
&= G'(1+\kappa )\int _{D}\tilde{u} -u\,dx+G'(-1+\phi)\int _{\Omega_-}\tilde{u}-u\,dx.\notag
\end{align}
Observing that $G'(-1+\phi)\leq G'(1+\kappa )$, we recover
\begin{align*}
E(\tilde{u})-E(u)\leq G'(1+\kappa )\int_{D\cup \Omega_-}\tilde{u}-u\,dx=0,
\end{align*}
where the equality is a consequence of (ii) and $\omegaint \tilde{u}\,dx=\omegaint u\,dx$.
\end{proof}
\subsection{Lower bound in the critical regime}\label{ss:lowcr}
We need an improved lower bound in the critical regime. The idea is that we can get an additional term from the integral over the ``bulk phase.'' (This additional term is higher-order in the off-critical regime.) The strategy is the same as before: On the one hand, we establish in \eqref{m:lowcr} a lower bound involving $V$; on the other hand, we use this estimate to deduce a lower bound on the maximum energy gap of any admissible path $\gamma\in\A$. (As for the off-critical case, this lower bound is not vacuous. According to \cite{BGLN}, \cite{CCELM}, or proposition \ref{upperbound}, the energy barrier is well defined for any $\phi=\xi L^ {-d/(d+1)}$ with $\xi>\xi_d$.)
\begin{prop}[Lower bound, critical regime]\label{prop:lowmcrit}
For any $d\geq 2$, there exists $\epsilon_0(d)>0$ with the following property. Given the critical scaling
\begin{align}
  \phi=\xi L^{-d/(d+1)}\qquad\text{for }\xi\in(\xi_d,\infty),\label{rexi}
\end{align}
any $u\in X_\phi$ with $E(u)\leq E(\baru)+C_{\xi_d}\phi^{-d+1}$ and $V(u)\leq\epsilon_0 L^d$ satisfies
\begin{align}
    &E(u)-E(\baru)\notag\\
    &\geq C_1(\phi)V(u)^{(d-1)/d}-\phi C_2(\phi)V(u)\notag\\
    &\;\;+\frac{G''(-1+2\phi^{1/3})\phi^{d+1}}{2\xi^{d+1}}\left(\int_{\Omega}(u-\bar{u})\big(\chi_2(u)+\chi_3(u)\big)\,dx\right)^2,\label{m:lowcr}
\end{align}
where $\xi_d$, $C_{\xi}=C_{\xi}(d)$ are defined in \eqref{xid} and \eqref{Cxi} and $C_1,\,C_2$ are as in proposition \ref{prop:lowm}.
As a consequence, the energy barrier defined in \eqref{ebd} satisfies the lower bound
\begin{align}
 \delE&\geq  C_\xi\phi^{-d+1}+o(\phi^{-d+1}).\label{m:lowecr}
\end{align}
\end{prop}
\begin{proof}
We begin by establishing \eqref{m:lowcr}. In light of
\begin{align}
  E(\baru)\sim \phi^{-d+1}\quad\text{in the critical regime},\label{hrom}
\end{align}
the condition $E(u)\leq E(\baru)+C_{\xi_d}\phi^{-d+1}$ implies
\begin{align}
  E(u)\lesssim E(\baru).\label{ele}
\end{align}
As in the proof of proposition \ref{prop:lowm}, we observe that the integral of $e(u)$ over the support of $\chi_3$ and $\chi_2$ is estimated by \eqref{st2} and \eqref{stp3}, respectively (where in order to deduce \eqref{stp3}, we replace the energy bound \eqref{enbu} by the bound \eqref{ele}).
The estimate \eqref{m:lowcr} then follows directly from the  improved bound on the support of $\chi_1$:
\begin{align}\label{proofcritical1}
&\int_{\Omega}e(u)\chi_1(u)\,dx\notag\\
&\geq \frac{G''(-1+2\phi^{1/3})\phi^{d+1}}{2\xi^{d+1}}\left(\int_{\Omega}(u-\bar{u})\big(\chi_2(u)+\chi_3(u)\big)\,dx\right)^2.
\end{align}
To see \eqref{proofcritical1}, we use the strict convexity of $G$ on $(-\infty,-1+2\phi^{1/3}]$ to estimate
\begin{align}
  \int_\Omega e(u)\chi_1(u)\,dx&\geq \int_\Omega \Big(G(u)-G(\baru)-G'(\baru)(u-\baru)\Big)\chi_1(u)\,dx\notag\\
  &\geq \inf_{\tau\in(-\infty,-1+2\phi^{1/3})}G''(\tau)\frac{1}{2} \int_\Omega (u-\baru)^2\chi_1(u)\,dx\notag\\
  &= \frac{G''(-1+2\phi^{1/3})}{2}\int_\Omega (u-\baru)^2\chi_1(u)\,dx.\label{gcan}
\end{align}
From H\"older's inequality, the simplistic bound $\int_\Omega \chi_1(u)\,dx\leq L^d$, and
$$\omegaint (u-\bar{u})\big(\chi_1(u)+\chi_2(u)+\chi_3(u)\big)\,dx=0,$$
we deduce
\begin{align}
 \int_{\Omega}(u-\bar{u})^2\chi_1(u)\,dx &\geq L^{-d}\left(\int_\Omega(u-\bar{u})\chi_1(u)\,dx\right)^2\notag\\
 &= L^{-d}\left(\int_{\Omega}(u-\bar{u})\big(\chi_2(u)+\chi_3(u)\big)\,dx\right)^2.\label{uc}
\end{align}
Substituting \eqref{uc} into \eqref{gcan} and recalling \eqref{rexi} gives \eqref{proofcritical1}.

Now we deduce \eqref{m:lowecr} from \eqref{m:lowcr}. To this end, consider any continuous path $\gamma\in\A$ (where $\A$ is defined in \eqref{admis}). Without loss of generality, we need only consider paths such that
\begin{align*}
\max_{t\in[0,1]}E(\gamma(t))-E(\baru)\leq C_\xi\phi^{-d+1}.
 \end{align*}
The monotonicity of $C_\xi$ with respect to $\xi$ gives $\max_{t\in[0,1]}E(\gamma(t))-E(\baru)\leq C_\xi\phi^{-d+1}\leq C_{\xi_d}\phi^{-d+1}$, so that \eqref{m:lowcr} holds for all $t\in [0,1]$ such that $V(\gamma(t))\leq \epsilon_0L^d$.
As usual we rely on the continuity of $V$ on $X_\phi$ and $V(\baru)=0$. Also we remark that the right-hand side of \eqref{m:lowcr} is positive for small, positive $V$. Hence it suffices to argue that the right-hand side of \eqref{m:lowcr} takes on the value   $C_\xi\phi^{-d+1}+o(\phi^{-d+1})$ for some $0<V\leq \epsilon_0L^d$ smaller than the first strictly positive zero of the right-hand side of \eqref{m:lowcr}, which we note is at least of the order $\phi^{-d}\ll \epsilon_0L^d$. Combining these observations, it suffices to show that the right-hand side of \eqref{m:lowcr} takes on the value   $C_\xi\phi^{-d+1}+o(\phi^{-d+1})$ for some
\begin{align*}
  0<V\lesssim \phi^{-d}.
\end{align*}

We would like to transform the third term on the right-hand side of \eqref{m:lowcr} for ``intermediate'' values of $V$. We observe that
\begin{align*}
\int_{\Omega}u\chi_2(u)\,dx\geq-\int_{\Omega}\chi_2(u)\,dx \quad \mbox{and} \quad \int_{\Omega}u\chi_3(u)\,dx\geq (1-2\phi^{1/3})V(u),
\end{align*}
from which it follows that
\begin{align}\label{BunionC}
 \lefteqn{\int_{\Omega}(u-\bar{u})\big(\chi_2(u)+\chi_3(u)\big)\,dx}\notag\\
  &\geq -\phi \int_{\Omega}\chi_2(u)\,dx+(2-2\phi^{1/3}-\phi)V(u).
\end{align}
We  claim that $V(u)$ dominates $\phi\int_\Omega \chi_2(u)\,dx$ for $u\in X_\phi$ with
\begin{align}
E(u)\lesssim E(\baru)\quad\text{and}\quad  V(u)\gg \phi^{-d+1}.\label{reg2}
\end{align}
 Indeed, from $G(s)\geq \frac{1}{2}\phi^{2/3}$ on the support of $\chi_2$, we have
\begin{align*}
\int_\Omega \chi_2(u)\,dx \leq 2\phi^{-2/3}\int_\Omega G(u)\chi_2(u)\,dx\leq 2\phi^{-2/3}E(u)\overset{\eqref{ele}}\lesssim \phi^{-2/3}E(\baru),
\end{align*}
which, combined with \eqref{hrom}, implies
\begin{align}
\phi\int_{\Omega}\chi_2(u)\,dx\lesssim \phi^{-d+4/3}\ll \phi^{1/3}V(u)\label{hhrom}
\end{align}
for $V(u)\gg \phi^{-d+1}$.
It follows that the right-hand side of \eqref{BunionC} is positive for $\phi$ sufficiently small and hence
\begin{align}
\lefteqn{\left(\int_{\Omega}(u-\bar{u})\big(\chi_2(u)+\chi_3(u)\big)\,dx\right)^ 2}\notag\\
&\geq\left((2-2\phi^{1/3}-\phi)V(u)-
    \phi\int_\Omega\chi_2(u)\,dx\right)^2.\label{intmd}
\end{align}
We apply the elementary inequality $(a+b)^2\geq (1-\delta)a^2-\delta^{-1}b^2$
with $\delta=\phi^{1/3}$ to deduce
\begin{eqnarray}
 \lefteqn{\left((2-2\phi^{1/3}-\phi)V(u)-
    \phi\int_\Omega\chi_2(u)\,dx\right)^2}\notag\\
    &\geq & (1-\phi^{1/3})\Big((2-2\phi^{1/3}-\phi)V(u)\Big)^2-\phi^{-1/3}\Big(\phi\int_\Omega\chi_2(u)\,dx\Big)^2\notag\\
    &\overset{\eqref{hhrom}}\geq&
    (1-\phi^{1/3})\Big((2-2\phi^{1/3}-\phi)V(u)\Big)^2-\phi^{1/3}V(u)^2\notag\\
    &\geq&(1-2\phi^{1/3})\Big((2-2\phi^{1/3}-\phi)V(u)\Big)^2.\label{nomb}
\end{eqnarray}
Combining \eqref{m:lowcr}, \eqref{intmd}, and \eqref{nomb} implies that for $u$ satisfying \eqref{reg2}, we have
\begin{align}
  E(u)-E(\baru)&\geq C_1(\phi)V(u)^{(d-1)/d}-\phi C_2(\phi)V(u)\notag\\
  &\qquad+\frac{C_3(\phi)\phi^{d+1}}{\xi^{d+1}}V(u)^2,\label{nbd}
\end{align}
where $C_1,\,C_2$ are defined in \eqref{c1}, \eqref{c2} and
\begin{align}
C_3(\phi):=\frac{G''(-1+2\phi^{1/3})}{2}(2-2\phi^{1/3}-\phi)^2(1-2\phi^{1/3}).  \label{c3}
\end{align}
Letting $\nu:=\phi^{d}V$, we rewrite \eqref{nbd} as
\begin{align}
&E(u)-E(\baru)\notag\\
&\geq \phi^{-d+1}\left( C_1(\phi)\nu(u)^{(d-1)/d}-C_2(\phi)\nu(u)+C_3(\phi)\xi^{-(d+1)}\nu(u)^2\right).\label{romb2}
\end{align}
We view the right-hand side of \eqref{romb2} as a function of $\nu$ and, considering the behavior of $C_1,\,C_2,\,C_3$ for $\phi\downarrow 0$,  observe that
\begin{align}
  f_{\phi,\xi}(\nu):&=C_1(\phi)\nu^{(d-1)/d}-C_2(\phi)\nu+C_3(\phi)\xi^{-(d+1)}\nu^2\notag\\
  &=f_\xi(\nu)+o(1),\notag
\end{align}
where $f_\xi:\R^+\to\R$ is defined in \eqref{glambda}. We  use \eqref{romb2} and the behavior of $f_{\phi,\xi}$ to deduce a lower bound on the energy barrier. Recall the definitions of $\nu_\xi$, $C_\xi$ (cf.\ \eqref{Cxi}). Analogously, let $\nu_{\phi,\xi}$ denote the first strictly positive zero of $f_{\phi,\xi}$. We deduce from \eqref{romb2} that any $\gamma\in \A$ satisfies
\begin{align*}
  \max_{t\in[0,1]}E(\gamma(t))-E(\baru)&\geq \phi^{-d+1}\;\sup\Big\{f_{\phi,\xi}(\nu)\colon \phi\ll\nu\leq \nu_{\phi,\xi}\Big\}\\
  &=C_\xi\phi^{-d+1}+o(\phi^{-d+1}).
\end{align*}
\end{proof}
\section {Upper bounds}\label{s:up}
In this section, we develop an upper bound for the energy barrier $\delE$ by constructing a continuous
path that connects the uniform state $\bar{u}$ to a state of lower energy and estimating the maximum energy along the path. As explained in subsection~\ref{ss:lit}, the main building block of our construction is the construction of~\cite{CCELM}, in which the energy of a ``droplet state'' is estimated as a function of the radius of the droplet.
There are a few differences in our setting, however, since we need to keep more terms and since the relative size of the error terms in the off-critical scaling is not as straightforward as in the critical case. For completeness, we include the details.

The first ingredient is the hyperbolic tangent function
\begin{align}\label{v}
v(x)=-\tanh (x/\sqrt{2}),
\end{align}
which is a minimizer of the energy on $\R$ subject to $\pm 1$ boundary conditions, so that in particular
\begin{align*}
E(v)=\int_\mathbb{R}\frac{1}{2}v_x^2+G(v)\,dx=c_0,
\end{align*}
for $c_0$ defined in \eqref{c0}.

The next step is to modify $v$ so that it reaches $\pm 1$ at finite distance from the origin.
For $R>0$, one defines an odd function $v_R:\R \rightarrow \R$ such that
\begin{align}\label{v_R}
v_R(x) := \begin{cases} v(x) & \textrm{for}\quad |x| <R \\ -\sgn (x) & \textrm{for}\quad |x| > 2R, \end{cases} \qquad
\end{align}
with a smooth, monotone interpolation on $R\leq |x|\leq 2R$.

As explained in subsection~\ref{ss:lit}, the idea of \cite{BCK1}, which is also used in \cite{CCELM}, is to put \emph{part of the total mass $V_+$} defined in \eqref{vplus} into a droplet. Consider the fractional volume $\eta V_+$ for $\eta \in [0,1]$ and define the corresponding radius
\begin{align}\label{reta}
r_\eta:=\eta^{\frac{1}{d}} \left(\frac{\phi d}{2 \sigma _d}\right)^{\frac{1}{d}}L.
\end{align}
The main building block of our construction is a trial function of the form
\begin{align}\label{trialfunction}
u_\eta(x):=v_R(|x|-r_\eta)+\alpha(\eta),
\end{align}
where $R>0$ is to be specified and $\alpha (\eta)$ is a constant chosen to accommodate the mean constraint from \eqref{H-phi}.
The droplet state $u_\eta$ can be viewed as a ``fractional droplet."

While in \cite{CCELM} the idea of the fractional droplet is used to study the global energy minimizer, we observe below that the path of growing droplets parameterized by $\eta$ provides an energetically favorable path out of the basin of attraction of $\baru$. Our upper bounds take the following form.
\bprop\label{upperbound}[Upper bounds]
Consider $\A$  defined in \eqref{admis}.
In the off-critical regime \eqref{subcritical}, there exists a continuous path $\gamma\in\A$ such that
\begin{align}
\max_{t\in [0,1]}E(\gamma(t))-E(\baru)\leq C_*\phi ^{-d+1}+o(\phi^{-d+1}),\label{no.1}
\end{align}
where $C_*$ is given by \eqref{constant}.

In the critical regime with $\phi=\xi L^{-d/(d+1)}$ for $\xi>\xi_d$ defined by \eqref{xid}, there exists a continuous path $\gamma\in\A$ such that
\begin{align}
\max_{t\in [0,1]}E(\gamma(t))-E(\baru)\leq C_\xi\phi ^{-d+1}+o(\phi^{-d+1}),\label{no.2}
\end{align}
where $C_\xi$ is given by \eqref{Cxi}.
\eprop
We begin in subsection \ref{ss:lemup} by presenting (without proof) the lemmas that we will need in order to bound the energy of our constructions. Then in subsection \ref{ss:consup}, we use these estimates to prove proposition \ref{upperbound}. Finally in subsection \ref{ss:pfup}, we give the proofs of the lemmas.
\subsection{Lemmas for the upper bound constructions}\label{ss:lemup}
Our main goal is a good bound on the energy gap of the ``droplets functions'' $u_\eta$ described above, at least for droplets of \emph{moderate radius}. In order to connect these functions to $\baru$, we  need an elementary lemma that says that we can interpolate between $\baru$ and a ``moderately sized droplet'' while keeping the energy gap well below $\phi^{-d+1}$.
\blemma\label{lemma2}
There exists $C\in\R$ with the following property. Fix any $R\geq 1$ and let $w_R(x):=v_R(|x|-R)+\alpha$ with $\alpha$ chosen so that $\omegaint w_R(x)\,dx=(-1+\phi)L^d$.
For $\lambda \in [0,1]$, let $u_\lambda$ denote the convex combination
\begin{align*}
u_\lambda:=(1-\lambda)\unif+\lambda w_R.
\end{align*}
Then $u_\lambda\in X_\phi$ and
\begin{align*}
E(u_\lambda)\leq E(\unif)+CR^d.
\end{align*}
In particular, for every $\lambda \in [0,1]$ there holds
\begin{align*}
\Big(E(u_\lambda)-E(\unif)\Big)_+\ll \phi ^{-d+1}
\end{align*}
as long as $R\ll \phi^{-1+1/d}$.
\elemma
Now let us consider the droplets. Our first ingredient is an estimate of the constant $\alpha$ from \eqref{trialfunction}. The lemma
is a slight adaptation of \cite [lemma 2.1]{CCELM}.
\blemma\label{lemma1}
In the off-critical or critical regime, there exist constants $C, R_0<\infty$ with the following property. For any $R\geq R_0$ and $\reta \in [R, r_+]$, there holds
\begin{align}\label{integral1}
\int_\Omega v_R(|x|-r_\eta)\,dx=L^d(-1+\phi \eta)+r_\eta^{d-2}\left(C_1+O(e^{-R/C})\right)+\epsilon,
\end{align}
where $C_1:=(d-1) \sigmad \int_{-\infty}^\infty \left(\sgn (\xi)+v(\xi)\right)\xi \,d\xi>0$ and the error term is given by
\begin{align*}
\epsilon=
\begin{cases}
O(e^{-R/C}) & \text{for}\quad d=2, 3, \\
O(r_\eta^{d-4}) & \text{for}\quad d\geq 4.
\end{cases}
\end{align*}
Here $v$, $v_R$, and $\reta$ are given by \eqref{v}, \eqref{v_R}, and \eqref{reta}, respectively, and $r_+$ denotes the radius of a ball of volume $V_+$ defined in \eqref{vplus}.

As a result, in  the off-critical or critical regime, the constant $\alpha$ appearing in \eqref{trialfunction} satisfies
\begin{align}\label{alpha}
\alpha(\eta)=\phi (1-\eta)-\frac{C_1}{L^d}\reta^{d-2}\left(1+O(e^{-R/C})\right)-\frac{1}{L^d}O(\reta^{d-4}).
\end{align}
\elemma
\begin{rmrk}\label{1-eta}
Note that for $1-\eta\gg \phi^{-2/d}L^{-2}$---which in both the off-critical and critical regimes is satisfied for $1-\eta \gg \phi ^2$---
the first term in \eqref{alpha} is dominant, i.e., $\phi (1-\eta)\gg \reta^{d-2}L^{-d}$. In what follows, it will suffice to restrict to $\eta$ values such that $1-\eta\gg\phi^2$, and \eqref{alpha} will help in estimating the energy of $u_\eta$.
\end{rmrk}
We turn to an estimate of the energy of $u_\eta$. The following lemma is a slight modification of \cite[lemma 2.2]{CCELM}.
\blemma\label{l:ueta}
There exist constants $C, R_0<\infty$ so that, for any $R\geq R_0$ and $\reta\geq R$, $1-\eta\gg\phi^2$, the energy of $u_\eta$ in the off-critical or critical regime satisfies
\begin{align}
E(u_\eta)&=c_0\sigmad \reta^{d-1}\left(1+O(e^{-R/C})\right)+ \phi^2L^d(1-\eta)^2-\phi^3L^d(1-\eta)^3\notag\\
&+\frac{\phi^4L^d}{4}(1-\eta)^4 +O\left(\phi (1-\eta)\reta^{d-2}\right)+O\left(\phi^2(1-\eta)^2\reta^{d-1}\right)\notag\\
&+O\left(\phi^4\eta(1-\eta)^3L^d\right)+O(\reta^{d-3}).\label{s.six}
\end{align}
Here $\reta$ and $u_\eta$ are given by \eqref{reta} and \eqref{trialfunction}, respectively.
\elemma
\begin{rmrk}\label{rem:eunif}
Recall from formula \eqref{E(unif)} that $E(\bar{u})=\phi ^2 L^d-\phi ^3 L^d+\frac{\phi ^4}{4}L^d$, so that \eqref{s.six} implies, for all $r_\eta\geq R$ and $\eta$ such that $1-\eta\gg \phi^2$, that
\begin{align}
E(u_\eta)-E(\baru)&=c_0\sigmad \reta^{d-1}\left(1+O(e^{-R/C})\right)+ \phi^2L^d(-2\eta+\eta^2)\notag\\
&\quad -\phi^3L^d(-3\eta+3\eta^2-\eta^3)+\frac{\phi^4L^d}{4}(-4\eta+6\eta^2-4\eta^3+\eta^4) \notag\\
&\quad +O\left(\phi (1-\eta)\reta^{d-2}\right)+O\left(\phi^2(1-\eta)^2\reta^{d-1}\right)\notag\\
&\quad +O\left(\phi^4\eta(1-\eta)^3L^d\right)+O(\reta^{d-3}).\label{s.2}
\end{align}
For consistency with the notation we used in the lower bounds, we substitute the definition \eqref{reta} of $r_\eta$ and reexpress this estimate in terms of the volume
\begin{align}
    V_\eta:&=\frac{\sigma_d}{d}r_\eta^d \overset{\eqref{reta}}=\eta\frac{\phi L^d}{2}\overset{\eqref{vplus}}=\eta V_+.\label{veta}
\end{align}
This leads to the observation that for $R\gg 1$ and $V_\eta$ satisfying
\begin{align}
  \frac{\sigma_d}{d}R^d\leq V_\eta \ll V_+,\label{vlim2}
\end{align}
one has in the off-critical regime that
\begin{align}
E(u_\eta)-E(\baru)\leq  \bar{C}_1 V_\eta^{(d-1)/d} -4\phi V_\eta +o(\phi^{-d+1}), \label{mainoff}
\end{align}
and that for $R\gg 1$ and $V_\eta$ satisfying
\begin{align}
  \frac{\sigma_d}{d}R^d\leq V_\eta\quad\text{and}\quad V_+-V_\eta\gg \phi^2 V_+,\label{vlim}
\end{align}
one has in the critical regime that
\begin{align}
  E(u_\eta)-E(\baru)\leq  \bar{C}_1 V_\eta^{(d-1)/d} -4\phi V_\eta+\frac{4\phi^{d+1}V_\eta^2}{\xi^{d+1}} +o(\phi^{-d+1}), \label{maincr}
\end{align}
where $\bar{C}_1$ is defined in \eqref{barc1}. We will derive our control of the energy barrier from \eqref{mainoff} and \eqref{maincr}.
\end{rmrk}
\subsection{Proof of proposition  \ref{upperbound}}\label{ss:consup}
\begin{proof}[Proof of proposition \ref{upperbound}]
We use the construction from lemma \ref{lemma2} for the first part of the path and the construction from lemma \ref{l:ueta} for the second part of the path. According to lemma \ref{lemma2}, as long as $R\ll \phi ^{-1+1/d}$, the contribution from the first part of the path is negligible with respect to the right-hand side of \eqref{no.1}, \eqref{no.2}, respectively. Hence we choose $R$ to satisfy $1\ll R\ll \phi ^{-1+1/d}$, and our main task is to analyze \eqref{mainoff} and \eqref{maincr} in the off-critical and critical regimes, respectively.

We begin with the off-critical regime. Using $V_+\gg \phi^{-d}$ in the off-critical regime, the condition \eqref{vlim2} and estimate \eqref{mainoff} can be reexpressed in terms of the rescaled volume $\nu_\eta=\phi^{d}V_\eta$ in the following way: For any $C<\infty$ and for all $\nu_\eta$ with
\begin{align}
 \frac{\sigma_d}{d}R^d\phi^d\leq \nu_\eta\leq C,\label{nulim}
\end{align}
we have for $\phi\ll 1$ that
\begin{align}
E(u_\eta)-E(\baru)\leq  \phi^{-d+1}f_\infty(\nu_\eta)+o(\phi^{-d+1}),\label{mainoff2}
\end{align}
where $f_\infty$ is defined in \eqref{f0} (and is independent of $\phi$).
Notice that $R^d\phi^d\ll 1$ (by choice of $R$). To deduce \eqref{no.1}  from \eqref{mainoff2}, it therefore suffices to check that
\begin{enumerate}
  \item[(i)] there exists $\nu_->0$  such that $f_\infty(\nu_-)<0$,
  \item[(ii)] $\sup\Big\{ f_\infty(\nu)\colon 0\leq \nu\leq \nu_- \Big\}\leq C_*$.
\end{enumerate}
Indeed,
\eqref{mainoff2} and (i) imply that there exists a point $\nu_-$ satisfying \eqref{nulim} and a corresponding function $u_{\eta_-}$ along our constructed path such that $E(u_{\eta_-})<E(\baru)$, while \eqref{mainoff2} and (ii) imply that the energy along the second part of the path until reaching $u_{\eta_-}$ stays below $C_*\phi^{-d+1}+o(\phi^{-d+1})$.
The observations (i) and (ii) concerning $f_\infty$ are elementary (and were already made in  subsection \ref{ss:lit}).
This concludes the proof of \eqref{no.1}.

We now consider the critical regime. Using $V_+=\phi^{-d}\xi^{d+1}/2$ in the critical regime, \eqref{vlim} and \eqref{maincr} can be rewritten in terms of the rescaled volume $\nu_\eta=\phi^{d}V_\eta$ in the following way: For $\nu_\eta$ satisfying
\begin{align}
 \frac{\sigma_d}{d}R^d\phi^d\leq \nu_\eta\quad\text{and}\quad \frac{\xi^{d+1}}{2}-\nu_\eta\gg \phi^2,  \label{nulim3}
\end{align}
we have
\begin{align}
E(u_\eta)-E(\baru)\leq  \phi^{-d+1}f_\xi(\nu_\eta)+o(\phi^{-d+1}),\label{nono}
\end{align}
where $f_\xi$ is defined in \eqref{glambda} (and is independent of $\phi$).
By choice of $R$,  $R^d\phi^d\ll 1$ as in the off-critical regime.
To deduce \eqref{no.2} from \eqref{nulim3} and \eqref{nono}, it suffices to check that
\begin{enumerate}
  \item[(i')] there exists $0<\nu_-<\frac{1}{2}\xi^{d+1}$ such that $f_\xi(\nu_-)<0$,
  \item[(ii')] $\sup\Big\{ f_\xi(\nu)\colon 0\leq \nu\leq \nu_-\Big\}\leq C_\xi$.
\end{enumerate}
Condition (ii') is automatically satisfied by the definition \eqref{Cxi} of $C_\xi$.
To check condition (i'), we write $f_\xi$ as the product
\begin{align*}
  f_\xi(\nu)=\nu\big(\bar{C}_1\nu^{-1/d}-4+4\xi^{-(d+1)}\nu\big).
\end{align*}
Defining $g_\xi:=\bar{C}_1\nu^{-1/d}-4+4\xi^{-(d+1)}\nu$, we observe via elementary calculus that $\lim_{\nu\downarrow 0}g_\xi(\nu)=\infty$ and $g_\xi(\nu_m)<0$ where $\nu_m$ denotes the local minimum
\begin{align*}
  \nu_m:=\frac{c_0^{d/(d+1)}\sigma_d^{1/(d+1)}}{4^{d/(d+1)}d^{1/(d+1)}}\xi^d.
\end{align*}
Clearly $\nu_m<\frac{1}{2}\xi^{d+1}$ precisely if
\begin{align}
  \xi>\frac{2c_0^{d/(d+1)}\sigma_d^{1/(d+1)}}{4^{d/(d+1)}d^{1/(d+1)}}.\label{xibig}
\end{align}
Since $\xi>\xi_d$ (defined in \eqref{xid}), it enough to check whether $\xi_d$ satisfies \eqref{xibig}, which it does if and only if
$d>1$.
We deduce that condition (i') holds and hence, \eqref{no.2} is established.
\end{proof}
\subsection{Proofs of lemmas}\label{ss:pfup}
\begin{proof}[Proof of lemma \ref{lemma2}]
The fact that $\omegaint u_\lambda\,dx=(-1+\phi)L^d$ follows immediately from linearity of the integral and the choice of $w_R$. Let $0<\tilde{r}<+\infty$ be defined through $w_R(x)=-1+\phi$ for $|x|=\tilde{r}$. Using the fact that $0<\alpha <\phi$ (c.f. \eqref{alpha}) we can easily see that $R<\tilde{r}<3R$. Consequently, for $\lambda \in [0,1]$ we have
\begin{align}\label{stupid1}
\int_{B_{\tilde{r}}(0)}G(u_\lambda)dx\lesssim |B_{\tilde{r}}(0)|\sim R^d,
\end{align}
where $B_{\tilde{r}}(0)$ denotes the open ball of radius $\tilde{r}$ that is centered at the origin.
Also
\begin{align}\label{stupid2}
\int_{\Omega\setminus B_{\tilde{r}}(0)}G(u_\lambda)dx\leq \int_{\Omega\setminus B_{\tilde{r}}(0)}G(\unif)dx\leq E(\unif),
\end{align}
since $G(u_\lambda)\leq G(-1+\phi)$ on $\Omega\setminus B_{\tilde{r}}(0)$.
The gradient term of the energy of $w_R$ is
\begin{align*}
\omegaint|\nabla w_R|^2dx&=\omegaint |\nabla v_R(|x|-R)|^2dx=\sigmad \int_0^\infty \xi^{d-1}(v_R'(\xi-R))^2d\xi \\ & \leq \sigmad R^{d-1}\int_{-\infty}^\infty (v_R'(\xi))^2\left|1+\frac{\xi}{R}\right|^{d-1}d\xi.
\end{align*}
The last integral is bounded independently of $R$. To see why, first note that the integrand vanishes outside the interval $[-2R,2R]$, therefore
\begin{align*}
\int_{-\infty}^\infty (v_R'(\xi))^2\left|1+\frac{\xi}{R}\right|^{d-1}d\xi&\leq C \int_{-\infty}^\infty (v_R'(\xi))^2d\xi\\
&= C\left(\int_{-R}^R(v'(\xi))^2d\xi+2\int_R^{2R}(v_R'(\xi))^2d\xi\right)\\
& \leq C \left(\int_{-\infty}^\infty(v'(\xi))^2d\xi+\int_R^{2R}(v_R'(\xi))^2d\xi\right).
\end{align*}
Since $v'(\xi)$ decays exponentially, the first integral inside the brackets is finite. Also, since on $[R,2R]$ we have $v_R'(\xi)=O(e^{-R/C})$, the second integral in the brackets above is also bounded independently of $R$.
We therefore have
\begin{align*}
\omegaint |\nabla w_R|^2dx \lesssim R^{d-1},
\end{align*}
which, in combination with \eqref{stupid1} and \eqref{stupid2}, implies that
\begin{align*}
E(u_\lambda)= \omegaint \lambda^2|\nabla w_R|^2+G(u_\lambda)\,dx\leq E(\unif)+O(R^d).
\end{align*}
\end{proof}
\begin{proof}[Proof of lemma \ref{lemma1}]
Let $B_{r_\eta}(0)$ denote the open ball of radius $\reta$ centered at the origin. Since $\reta \leq r_+ \ll L$, we have $B_{\reta} (0)\subset \Omega$. Note that
\begin{align*}
\omegaint \sgn(|x|-r_\eta)dx=L^d-2|B_{r_\eta}(0)|=L^d(1-\phi \eta).
\end{align*}
Consequently, we can write
\begin{align}\label{integral2}
\omegaint v_R(|x|-r_\eta)dx=L^d(-1+\phi \eta)+\omegaint \sgn(|x|-r_\eta)+v_R(|x|-r_\eta)\,dx.
\end{align}
Comparing \eqref{integral1} and \eqref{integral2}, it suffices to estimate the second term on the right-hand of side \eqref{integral2}, which we decompose as
\begin{align*}
\omegaint \sgn(|x|-r_\eta)+v_R(|x|-r_\eta)\,dx = I_1+I_2,
\end{align*}
where
\begin{align*}
I_1:=\int_{\Omega \cap \{|x|>2r_\eta\}}1+v_R(|x|-r_\eta)\,dx
\end{align*}
and
\begin{align*}
I_2:=\int_{\{|x|<2r_\eta\}}\sgn(|x|-r_\eta)+v_R(|x|-r_\eta)\,dx.
\end{align*}
Since for $|x|>2r_\eta$ we have
\begin{align*}
|1+v_R(|x|-r_\eta)|\leq e^{-(|x|+R)/C}
\end{align*}
for some $0<C<\infty$ that is independent of $R$, it follows that
\begin{align}\label{I_1}
I_1=O(e^{-R/C}).
\end{align}
Turning to $I_2$, we introduce polar coordinates in order to express
\begin{align*}
I_2 &=\sigmad \int_0^{2r_\eta}\Big(\sgn(\xi-\reta )+v_R (\xi-\reta)\Big)\xi^{d-1}d\xi \\ &= \sigmad \int_{-\reta}^{\reta}\Big(\sgn(\xi)+v_R (\xi)\Big)(\xi+\reta)^{d-1}d\xi.
\end{align*}
It is convenient to denote the right-hand side as $I_2'-I_2''$, where
\begin{align*}
I_2':=\sigmad \int_{-\infty}^\infty \Big(\sgn (\xi)+v_R(\xi)\Big)(\xi+\reta)^{d-1}d\xi
\end{align*}
and
\begin{align*}
I_2'':=\sigmad \int_{\R\setminus [-\reta,\reta]} \Big(\sgn (\xi)+v_R(\xi)\Big)(\xi+\reta)^{d-1}d\xi.
\end{align*}
We write $I_2'$ as
\begin{align*}
I_2'&=\sigmad \reta^{d-1}\int_{-\infty}^\infty \Big(\sgn (\xi)+v_R(\xi)\Big)\left(1+\frac{\xi}{\reta}\right)^{d-1}d\xi \\ & =\sigmad \reta^{d-1}\int_{-\infty}^\infty \Big(\sgn (\xi)+v_R(\xi)\Big)\sum_{k=0}^{d-1}\binom{d-1}{k}\left(\frac{\xi}{\reta}\right)^kd\xi.
\end{align*}
Using the fact that $\sgn (\xi)+v_R(\xi)$ is odd in $\xi$, we obtain
\begin{align*}
I_2'=\sigmad \reta^{d-2}(d-1)\int_{-\infty}^\infty \Big(\sgn (\xi)+v_R(\xi)\Big)\xi d\xi +O(\reta^{d-4}),
\end{align*}
where the $O(\reta^{d-4})$ term appears only for $d\geq 4$. Furthermore, since
\begin{align*}
|v(\xi)-v_R(\xi)|\leq e^{-(|\xi|+R)/C},
\end{align*}
we can express $I_2'$ in terms of the $R$-independent profile $v$ as
\begin{align}\label{I_2'}
I_2'&=\sigmad \reta^{d-2}(d-1)\left(\int_{-\infty}^\infty \Big(\sgn (\xi)+v(\xi)\Big)\xi d\xi+O(e^{-R/C})\right)+O(\reta^{d-4})\nonumber \\ &=C_1\reta^{d-2}\Big(1+O(e^{-R/C})\Big)+O(\reta^{d-4}),
\end{align}
again with the $O(\reta^{d-4})$ term appearing only for $d\geq 4$.

Finally, using the fact that $|\sgn(\xi)+v_R(\xi)|\leq e^{-|\xi|/C}$, we similarly obtain
\begin{align}\label{I_2''}
I_2''=O(e^{-R/C})
\end{align}
from which the result follows by combining \eqref{I_1} with \eqref{I_2'} and \eqref{I_2''}, and noting that, for $d\geq 4$, the $O(e^{-R/C})$ term can be absorbed into the $O(\reta^{d-4})$ error term.
\end{proof}
\begin{proof}[Proof of lemma \ref{l:ueta}]
In the proof we will abbreviate by writing $v_R$ instead of $v_R(|x|-\reta)$. We decompose the energy of $u_\eta$ as
\begin{align}\label{energy decomposition}
E(u_\eta)&=\omegaint \frac{1}{2}|\nabla v_R|^2+G(v_R)\,dx+\alpha \omegaint G'(v_R)\,dx+\frac{\alpha ^2}{2}\omegaint G''(v_R)\,dx \nonumber \\& \qquad \qquad +\frac{\alpha ^3}{3!}\omegaint G'''(v_R)\,dx+\frac{\alpha^4}{4!}\omegaint G^{(4)}(v_R)\,dx \nonumber \\&=:I_0+I_1+I_2+I_3+I_4.
\end{align}
We now estimate each of the terms in \eqref{energy decomposition}.

\underline{Estimate of $I_0$}.
Introducing polar coordinates and using the compact support of $|\nabla v_R|$ and $G(v_R)$, we write
\begin{align}
I_0&=\omegaint \frac{1}{2}|\nabla v_R|^2+G(v_R)\,dx\notag\\
&=\int_0^\infty \sigmad \xi^{d-1}\left(\frac{1}{2}(v_R'(\xi-\reta))^2+G(v_R(\xi-\reta))\right)d\xi \nonumber \\ &=\int_{-\reta}^\infty \sigmad \left(\frac{1}{2}(v_R'(\xi))^2+G(v_R(\xi))\right)(\reta+\xi)^{d-1}d\xi \nonumber \\
& = \int_{-\infty}^\infty \sigmad \left(\frac{1}{2}(v_R'(\xi))^2+G(v_R(\xi))\right)(\reta+\xi)^{d-1}d\xi\notag\\
&\quad+\int_{-\infty}^{-\reta} \sigmad \left(\frac{1}{2}(v_R'(\xi))^2+G(v_R(\xi))\right)(\reta+\xi)^{d-1}d\xi\notag\\
&=:I_0'-I_0''.\label{I_0}
\end{align}
Using again that $v_R$ is odd and exponentially close to $v$ for large $|x|$, we estimate
\begin{align}\label{I_0'}
I_0' &=\sigmad \reta ^{d-1}\int_{-\infty}^\infty \left(\frac{1}{2}(v_R'(\xi))^2+G(v_R(\xi))\right)\left(1+\frac{\xi}{\reta}\right)^{d-1} d\xi \nonumber \\ & =\sigmad \reta^{d-1}\int_{-\infty}^\infty \left(\frac{1}{2}(v'(\xi))^2+G(v(\xi))+O\left(e^{-(R+|\xi|)/C}\right)\right)\left(1+\frac{\xi}{\reta}\right)^{d-1}d\xi \nonumber \\ & =c_0\sigmad \reta^{d-1}+O\left(\reta^{d-1}e^{-R/C}\right)+O(\reta^{d-3}),
\end{align}
where the $O(\reta^{d-3})$ term appears only when $d\geq 3$.
We also have
\begin{align*}
|I_0''|&=\left|\sigmad \reta^{d-1}\int _{\reta}^\infty \left(\frac{1}{2}(v_R'(\xi))^2+G(v_R(\xi))\right)\left(1-\frac{\xi}{\reta}\right)^{d-1} d\xi \right| \\
& \lesssim \sigmad \int_0^\infty e^{-(\xi+R)/C}\xi^{d-1}d\xi=O(e^{-R/C}),
\end{align*}
so that absorbing this into \eqref{I_0'} and recalling \eqref{I_0} we conclude that
\begin{align}
I_0=c_0\sigmad \reta^{d-1}\left(1+O(e^{-R/C})\right)+O(\reta^{d-3}).\label{s.one}
\end{align}
\underline{Estimate of $I_1$}.
As in the proof of lemma \ref{lemma1}, we estimate
\begin{align*}
\omegaint G'(v_R)dx=(d-1)C\sigmad \reta^{d-2}\left(1+O(e^{-R/C})\right)+O(\reta^{d-4}),
\end{align*}
where $C=\int_{-\infty}^\infty \left(v^3(\xi)-v(\xi)\right)\xi d\xi$ and the $O(\reta^{d-4})$ error term is present only for $d\geq 4$.
We thus obtain, with the help of \eqref{alpha}, the estimate
\begin{align}
I_1&=\alpha \omegaint G'(v_R)dx=O( \phi (1-\eta)\reta^{d-2})+O(\phi \reta^{d-4}).\label{s.two}
\end{align}
\underline{Estimate of $I_2$}.
We note that
\begin{align}
I_2&=\frac{\alpha ^2}{2}\omegaint G''(v_R)dx=\frac{\alpha^2}{2}\omegaint 3v_R^2-1\,dx\notag\\
&=\alpha^2 L^d+\frac{3\alpha ^2}{2}\omegaint v_R^2-1\,dx,\label{I_2}
\end{align}
and we express the integral on the right-hand side as
\begin{align}
\omegaint v_R^2-1\,dx&=\int_0^\infty \sigmad (v_R^2(\xi-\reta)-1)\xi^{d-1}\,d\xi\notag\\
&=\int_{-\infty}^\infty \sigmad (v_R^2(\xi-\reta)-1)\xi^{d-1}\,d\xi \nonumber \\ &-\int_{-\infty}^{0} \sigmad (v_R^2(\xi-\reta)-1)\xi^{d-1}\,d\xi=:I_2'-I_2''.\label{some integral}
\end{align}
Changing variables, we express $I_2'$ and $I_2''$ as
\begin{align*}
I_2'=\int_{-\infty}^\infty \sigmad (v_R^2(\xi)-1)(\reta+\xi)^{d-1}\,d\xi,
\end{align*}
and
\begin{align*}
I_2''=\int_{-\infty}^{-\reta} \sigmad (v_R^2(\xi)-1)(\reta+\xi)^{d-1}d\xi.
\end{align*}
As we have checked in the proof of lemma \ref{lemma1}, $|I_2''|=O(e^{-R/C})$. On the other hand, we observe that
\begin{align}
I_2'&=\sigmad \reta^{d-1}\int_{-\infty}^\infty \left(v_R^2(\xi)-1\right)\left(1+\frac{\xi}{\reta}\right)^{d-1}d\xi\notag\\
&= \sigmad \reta^{d-1}\int_{-\infty}^\infty \left(v_R^2(\xi)-1\right)\sum_{k=0}^{d-1}\binom{d-1}{k}\left(\frac{\xi}{\reta}\right)^kd\xi \nonumber \\
&=-C\sigmad \reta^{d-1}\left(1+O(e^{-R/C})\right)+O(\reta^{d-3})= O(\reta^{d-1}) ,\notag
\end{align}
with $C=\int_\R \left( 1-v^2 (\xi)\right)d\xi>0$ and where the error term $O(\reta^{d-3})$ is included only for $d\geq 3$.
Inserting these two bounds into \eqref{some integral} and then back into
\eqref{I_2} together with the estimate \eqref{alpha} of $\alpha$ yields the estimate
\begin{align}
I_2=L^d\phi^2(1-\eta)^2+O\left(\phi (1-\eta)\reta^{d-2}\right)+O\left(\phi^2(1-\eta)^2 \reta^{d-1}\right).\label{s.three}
\end{align}
\underline{Estimate of $I_3$}.
Using lemma \ref{lemma1} and the estimate \eqref{alpha}, the quantity $I_3=\alpha^3 \omegaint v_Rdx$ is estimated as
\begin{align}
I_3=-L^d \phi^3 (1-\eta)^3 + O\left(\phi^2 (1-\eta)^2\reta^{d-2}\right)+O\left(\phi^4 \eta (1-\eta)^3 L^d\right).\label{s.four}
\end{align}
\underline{Estimate of $I_4$}.
Invoking the estimate \eqref{alpha} on $\alpha$ one more time, along with Remark \ref{1-eta} and the assumption $1-\eta \gg \phi^2$ gives
\begin{align}
I_4=\frac{L^d}{4}\alpha^4=\frac{L^d}{4}\phi^4 (1-\eta)^4+O\left(\phi^3(1-\eta)^3\reta^{d-2}\right).\label{s.five}
\end{align}
Inserting  the estimates \eqref{s.one}, \eqref{s.two}, \eqref{s.three}, \eqref{s.four}, and \eqref{s.five} into the decomposition \eqref{energy decomposition} yields \eqref{s.six}.
\end{proof}
\section{$\Gamma$-convergence of the rescaled energy gap}\label{Gconv}
In this section we study the leading order behavior of the rescaled energy gap $\phi$
as $(\phi,L)  \to (0,\infty)$ in the critical regime. (For the off-critical regime, we recall remark \ref{rem:gamoff}.) The normalization by $\phi^{-d+1}$ is selected by theorem \ref{theorem}. The proof of theorem \ref{theorem} also suggests that the functions of interest in $X_\phi$ satisfy $u\approx +1$ on sets of volume $\sim \phi^{-d}$.
Rescaling  space by a factor of $\phi$, we rewrite the rescaled energy gap from  \eqref{rescaled1} as
\begin{align}\label{rescaled2}
\int_{\Omega_{\phi, L}}\frac{\phi}{2}|\nabla u|^2+\frac{1}{\phi}\Big(G(u)-G(-1+\phi)\Big)\,dx,
\end{align}
where $\Omega_{\phi, L}:=[-\phi L/2,\phi L/2]^d$.

Our proof of the $\Gamma$-convergence is largely a ``sharp interface version'' of the proofs of propositions \ref{prop:lowm}, \ref{prop:lowmcrit}, and \ref{upperbound}. For completeness, we give the details (although in somewhat abbreviated format since the logical arguments have been made above). We will use the sharp partition of $\Omega_{\phi, L}$ into the sets
\begin{align}
A^\kappa_\phi:&=\{x \in \Omega_{\phi, L}: u_\phi (x)<-1+\kappa\},\label{Akappaphi}\\
B^\kappa_\phi:&=\{x \in \Omega_{\phi, L}:-1+\kappa \leq u_\phi (x)\leq 1-\kappa\},\label{Bkappaphi}\\
C^\kappa_\phi:&=\{x \in \Omega_{\phi, L}: u_\phi (x)>1-\kappa\},\label{Ckappaphi}
\end{align}
as well as the smooth partition of unity defined in \eqref{chi1}-\eqref{chi3}, for $\kappa \in (0,1/2)$ that in this part will be taken  to be a constant \emph{that is fixed with respect to $\phi$}.
\begin{proof}[Proof of theorem \ref{Gamma}]
We remark for reference below that
\begin{align}
  \frac{1}{\phi}\int_{\Omega_{\phi, L}}G(-1+\phi)\,dx\overset{\eqref{lim}}\longrightarrow\xi^{d+1}\quad\text{as }\phi\downarrow 0,\label{bul2}\\
  \phi|\omegphil|\overset{\eqref{lim}}\longrightarrow\xi^{d+1}\quad\text{as }\phi\downarrow 0 \label{bul3}.
\end{align}
Throughout the proof, for a given limit function $u_0$, let
\begin{align*}
  C:=\{x\in \R^ d\colon u_0(x)=+1\}.
\end{align*}

\emph{Step 1: Lower semicontinuity}
\medskip

Let $u_0\in -1 + L^p(\mathbb{R}^d)$ for $p\in (1,\infty)$ and suppose $u_\phi \to u_0$ in $-1 + L^p(\mathbb{R}^d)$ as $(\phi,L) \to (0,\infty)$. We need to show that
\begin{align}\label{Gproof-1}
\liminf _{\phi,L}\ephi (u_\phi) \geq \mathcal{E}^\xi_0(u_0).
\end{align}

Suppose that the condition $u_0=\pm1$ a.e. does not hold. Then $G(u_0)> 0$ on a set of positive measure and, in particular, there is a compact set $K\subset\R^ d$ such that
$
  \int_K G(u_0)\,dx>0.
$
Using $L^p(K)$ convergence, we may assume (up to a subsequence) that $u_\phi$ converges to $u_0$ almost everywhere on $K$, and hence by Fatou's lemma we obtain
\begin{align*}
\liminf_{\phi,L}\int_{\Omega_{\phi, L}} G(u_\phi)\,dx\geq \liminf_{\phi,L}\int_{K} G(u_\phi)\,dx>0.
\end{align*}
Combining this with \eqref{bul2} yields
\begin{align}
\liminf_{\phi,L}\ephi (u_\phi) & \geq \liminf_{\phi,L}\frac{1}{\phi}\int_{\Omega_{\phi, L}} G(u_\phi)-G(-1+\phi)\,dx\notag \\
&=\liminf_{\phi,L}\frac{1}{\phi}\int_{\Omega_{\phi, L}}   G(u_\phi)\,dx-\xi^{d+1}=+\infty, \notag
\end{align}
so that \eqref{Gproof-1} holds.

We now consider the case in which $u_0=\pm1$ a.e. We then have $u_0=-1+2\chi_C$, and since $u_0\in -1 + L^p(\mathbb{R}^d)$  it follows $|C|<\infty$. We also remark that we may assume without loss of generality that $u_\phi \leq 1+\kappa$. Indeed, if this is not the case, we can apply lemma \ref{lemma5} and replace $u_\phi$ by $\tilde{u}_\phi \leq 1+\kappa$ such that $\ephi (\tilde{u}_\phi)\leq \ephi (u_\phi)$. It is straightforward to check that the function $\tilde{u}_\phi$ constructed in the lemma also satisfies
\[
\|\tilde{u}_\phi-u_0\|^p_{L^p(\mathbb{R}^d)}\leq \|u_\phi-u_0\|^p_{L^p(\mathbb{R}^d)}+\phi^p |\Omega_{\phi,L}|.
\]
Using \eqref{bul3} and the fact $u_\phi \to u_0$ in $-1 + L^p(\mathbb{R}^d)$, we conclude that $\tilde{u}_\phi \to u_0$ in $-1 + L^p(\mathbb{R}^d)$.

Let us first assume that $C$ is of finite perimeter. The proof resembles our proofs of propositions \ref{prop:lowm}, \ref{prop:lowmcrit}. We use the mean constraint to write
\begin{align}\label{Gproof0}
\ephi (u_\phi)&=\int_{\Omega_{\phi, L}}  \frac{\phi}{2}|\nabla u_\phi|^2+\frac{1}{\phi}\Big(G(u_\phi)-G(-1+\phi)\notag\\
&\qquad\qquad-G'(-1+\phi)(u_\phi-(-1+\phi))\Big)\,dx \notag \\
&=:\int_{\Omega_{\phi, L}}  e_\phi(u_\phi)\,dx=\int_{\Omega_{\phi, L}}  e_\phi(u_\phi)(\chi_1(u_\phi)+\chi_2(u_\phi)+\chi_3(u_\phi))\,dx,
\end{align}
and we consider separately each of the three integrals that appear on the right-hand side of \eqref{Gproof0}. We recall that $\kappa\in (0,1/2)$ is fixed, and we assume without loss of generality that
\begin{align}
\phi\leq \kappa^3. \label{phik}
\end{align}

We turn first to the integral of $e_\phi(u_\phi)\chi_1(u_\phi)$. Using convexity of $G$ on $(-\infty,-1+2\kappa)$ and $\inf_{(-\infty,-1+2\kappa)}G''=G''(-1+2\kappa)$, we estimate as in \eqref{gcan} to obtain
\begin{equation}\label{pos1}
\int_{\Omega_{\phi, L}}  e_\phi(u_\phi)\chi_1(u_\phi)\,dx \geq  \frac{G''(-1+2\kappa)}{2\phi} \int_{\Omega_{\phi, L}}  (u_\phi-(-1+\phi))^2\chi_1(u_\phi)\,dx.
\end{equation}
We treat $|C|=0$ and $|C|>0$ separately. On the one hand, if $|C|=0$, we deduce from \eqref{pos1} that
\begin{align*}
\int_{\Omega_{\phi, L}}  e_\phi(u_\phi)\chi_1(u_\phi)\,dx\geq 0.
\end{align*}
On the other hand, if $|C|>0$, we use H\"older's inequality and the mean constraint as in \eqref{uc} to deduce from \eqref{pos1} that
\begin{align}
&\int_{\Omega_{\phi, L}} e_\phi(u_\phi)\chi_1(u_\phi)\,dx \notag \\
&\geq \frac{G''(-1+2\kappa)}{2\phi|\Omega_{\phi, L}|}\left(\int_{\Omega_{\phi, L}}(u_\phi-(-1+\phi))\chi_1(u_\phi)\,dx\right)^2 \notag \\
& = \frac{G''(-1+2\kappa)}{2\phi|\Omega_{\phi, L}|}\left(\int_{\Omega_{\phi, L}}(u_\phi-(-1+\phi))(\chi_2(u_\phi)+\chi_3(u_\phi))\,dx\right)^2.\label{Gproof1}
\end{align}
From
\begin{align*}
\int_{\Omega_{\phi, L}}u_\phi \chi_2(u_\phi)\, dx\geq -|B^\kappa_\phi| \quad \mbox{and} \quad \int_{\Omega_{\phi, L}}u_\phi \chi_3(u_\phi) \,dx \geq (1-2\kappa)|C^\kappa _\phi|,
\end{align*}
we deduce
\begin{align}\label{Gproof2}
\int_{\Omega_{\phi, L}}(u_\phi-(-1+\phi))(\chi_2(u_\phi)+\chi_3(u_\phi))\,dx \geq -|B^\kappa _\phi|+(2-2\kappa-\phi)|C^\kappa_\phi|.
\end{align}
Noting that
\begin{align*}
\|u_\phi-u_0\|^p_{L^p(\mathbb{R}^d)}&\geq \int_{B^\kappa_\phi}|u_\phi-u_0|^p\,dx \geq \kappa^p |B^\kappa_\phi|,\\
\|u_\phi-u_0\|^p_{L^p(\mathbb{R}^d)}&\geq \kappa^p |C\setminus C^\kappa_\phi|+(1-\kappa)^p|C^\kappa_\phi \setminus C|\geq \kappa^p |C\triangle C^\kappa _\phi|,
\end{align*}
we observe from the convergence of $u_\phi-u_0$ to zero in $L^p$ that
\begin{align}
  |B^\kappa_\phi|\to 0,\qquad |C^\kappa_\phi|\to |C|\label{bc}
\end{align}
as $(\phi,L)\to (0,\infty)$. Consequently, given $|C|>0$, the right-hand side of \eqref{Gproof2} is  nonnegative for $\phi$ small  and from \eqref{bul3}, \eqref{Gproof1}, and \eqref{bc} we obtain
\begin{align}\label{Gproof3}
\liminf_{\phi,L}\int_{\Omega_{\phi, L}}  e_\phi(u_\phi)\chi_1(u_\phi)\,dx \geq \frac{G''(-1+2\kappa)}{2\xi^{d+1}}(2-2\kappa)^2|C|^2.
\end{align}

For the integral of $e_\phi (u_\phi)\chi_3(u_\phi)$, we again use \eqref{mconv1} and
$u_\phi \leq 1+\kappa$ to estimate
\begin{align*}
\int_{\Omega_{\phi, L}}  e_\phi(u_\phi)\chi_3(u_\phi)\,dx  & \geq -\frac{1}{\phi}\,G'(-1+\phi)\int_{\Omega_{\phi, L}}(u_\phi+1)\chi_3(u_\phi)\,dx\\
& \geq -\frac{(2+\kappa)}{\phi}\,G'(-1+\phi)(|B^\kappa_\phi|+|C^\kappa_\phi|),
\end{align*}
so that
\begin{align}\label{Gproof8}
  \liminf_{\phi,L}\int_{\Omega_{\phi, L}}  e_\phi(u_\phi)\chi_3(u_\phi)\,dx   \overset{\eqref{bc}}\geq -2(2+\kappa)|C|.
\end{align}

We now turn our attention to the  the integral of $e_\phi (u_\phi)\chi_2(u_\phi)$. As in \eqref{r1}, we bound
\begin{align*}
\int_{\Omega_{\phi, L}}  e_\phi(u_\phi)\chi_2(u_\phi)\,dx\geq \int_{\Omega_{\phi, L}}\left(\frac{\phi}{2}|\nabla u_\phi|^2+\frac{1}{\phi}\left(G(u_\phi)-2G'(-1+\phi)\right)\right) \chi_2(u_\phi)\,dx,
\end{align*}
and we use \eqref{phik} to see that we can absorb the negative term with $8\kappa G(u_\phi)$. Similarly to in \eqref{per1}, we set $\tilde{G}(u)=(1-8\kappa)G(u)$ and estimate
\begin{align}\label{Gproof4}
\int_{\Omega_{\phi, L}}  e_\phi(u_\phi)\chi_2(u_\phi)\,dx&\geq \int_{\Omega_{\phi, L}}\left(\frac{\phi}{2}|\nabla u_\phi|^2+\frac{1}{\phi}\tilde{G}(u_\phi)\right)\chi_2(u_\phi)\,dx\notag\\
& \geq \int_{\Omega_{\phi, L}}\sqrt{2\tilde{G}(u_\phi)}\,\chi_2(u_\phi)|\nabla u_\phi|\,dx \notag \\
&  = \int_{-1+\kappa}^{1-\kappa}\sqrt{2\tilde{G}(t)}\,\chi_2(t)\text{Per}_{\Omega_{\phi,L}^\circ} (\{u_\phi >t\})\,dt\notag\\
& \geq \int_{-1+2\kappa}^{1-2\kappa}\sqrt{2\tilde{G}(t)}\,\text{Per}_{\Omega_{\phi,L}^\circ} (\{u_\phi >t\})\,dt \notag \\
&  \geq \einf_{\substack{-1+2\kappa\leq t \leq 1-2\kappa}}\text{Per}_{\Omega_{\phi,L}^\circ}(\{u_\phi >t\})\int_{-1+2\kappa}^{1-2\kappa}\sqrt{2\tilde{G}(t)}\,dt.
\end{align}
Here we have used $\text{Per}_{A^\circ}$ to stand for the perimeter in the interior of a set $A$. In the sharp interface limit, we can show that the infimum on the right-hand side of \eqref{Gproof4} converges to the perimeter of $C$ (in contrast to the bound \eqref{stp3} that we derived from \eqref{per1}). Indeed, choosing $t^\kappa_\phi \in [-1+2\kappa,1-2\kappa]$ such that
\begin{align*}
\text{Per}_{\Omega_{\phi,L}^\circ}(\{u_\phi>t^\kappa_\phi\})\leq \einf_{\substack{-1+2\kappa\leq t \leq 1-2\kappa}}\text{Per}_{\Omega_{\phi,L}^\circ} (\{u_\phi >t\})+\phi,
\end{align*}
we reexpress \eqref{Gproof4} as
\begin{align}\label{Gproof5}
\int_{\Omega_{\phi, L}}  e_\phi(u_\phi)\chi_2(u_\phi)\,dx&
\geq \left(\text{Per}_{\Omega_{\phi,L}^\circ}(\{u_\phi >t^\kappa_\phi\})-\phi\right)
\int_{-1+2\kappa}^{1-2\kappa}\sqrt{2\tilde{G}(t)}\,dt\notag\\
&=\left(\text{Per}_{\Omega_{\phi,L}^\circ}(\{u_\phi >t^\kappa_\phi\})-\phi\right)
(1-8\kappa)^{1/2}\int_{-1+2\kappa}^{1-2\kappa}\sqrt{2G(t)}\,dt.
\end{align}
Using
\begin{align*}
\|u_\phi-u_0\|^p_{L^p(\mathbb{R}^d)}\geq \kappa^p|C \triangle \{u_\phi >t^\kappa_\phi\}|,
\end{align*}
we notice that $\|u_\phi-u_0\|_{L^p(\mathbb{R}^d)}\to 0$  implies that $|C \triangle \{u_\phi >t^\kappa_\phi\}|\to 0$, from which it follows that $\|\chi _{\{u_\phi >t^\kappa_\phi\}}-\chi_C\|_{L^1(U)}\to 0$ for every open and bounded set $U \subset \mathbb{R}^d$. By the $L^1$-lower semicontinuity of the perimeter functional, and the fact that $U\subset {\Omega^\circ}_{\phi, L}$ for small $\phi$, we deduce
\begin{align*}
\liminf_{\phi,L}\text{Per}_{\Omega_{\phi,L}^\circ}(\{u_\phi >t^\kappa_\phi\})\geq \liminf_{\phi,L}\text{Per}_U(\{u_\phi >t^\kappa_\phi\})\geq \text{Per}_U(C),
\end{align*}
so that \eqref{Gproof5} becomes
\begin{align*}
\liminf_{\phi,L} \int_{\Omega_{\phi, L}}  e_\phi(u_\phi)\chi_2(u_\phi)\,dx\geq \text{Per}_U(C)(1-8\kappa)^{1/2}\int_{-1+2\kappa}^{1-2\kappa}\sqrt{2G(t)}\,dt,
\end{align*}
and upon taking the supremum over all open and bounded $U\subset \mathbb{R}^d$, and recalling that
\[
\sup_{U\subset \mathbb{R}^d} \{\text{Per}_U(C),\, U: \text{open and bounded}\}=\text{Per}(C),
\]
we see that, in fact,
\begin{align}
  \label{ug}
\liminf_{\phi,L} \int_{\Omega_{\phi, L}}  e_\phi(u_\phi)\chi_2(u_\phi)\,dx\geq \text{Per}(C)(1-8\kappa)^{1/2}\int_{-1+2\kappa}^{1-2\kappa}\sqrt{2G(t)}\,dt.
\end{align}
Substituting \eqref{Gproof3}, \eqref{Gproof8}, and \eqref{ug} into \eqref{Gproof0}, we obtain
\begin{align}\label{Gproof9}
\liminf_{\phi,L}\ephi (u_\phi)
& \geq \frac{G''(-1+2\kappa)}{2\xi^{d+1}}(2-2\kappa)^2|C|^2-2(2+\kappa)|C|\notag\\
&\qquad+\text{Per}(C)(1-8\kappa)^{1/2}\int_{-1+2\kappa}^{1-2\kappa}\sqrt{2G(t)}\,dt.
\end{align}
Letting $\kappa \to 0$, the right-hand side becomes $\mathcal{E}^\xi_0(u_0)$.

If $\text{Per}(C)=\infty$, the same argument implies
\begin{align*}
  \liminf_{\phi,L}\ephi (u_\phi)=\infty.
\end{align*}
\bigskip

\emph{Step 2: Recovery sequence}
\medskip

Here we show that for any $u_0 \in   -1+L^p(\mathbb{R}^d)$ there exists a sequence $\{u_\phi\}_{\phi>0} $ of functions in $-1+L^p(\mathbb{R}^d)$ such that
\begin{align}\label{Gconstruction-1}
\lim_{\phi,L}\|u_\phi -u_0\|_{L^p(\mathbb{R}^d)}=0 \quad \mbox{and} \quad \lim_{\phi,L}\ephi (u_\phi)\leq\mathcal{E}^\xi_0(u_0).
\end{align}
If $\mathcal{E}^\xi_0(u_0)=+\infty$, then \eqref{Gconstruction-1} is trivially satisfied by $u_\phi=u_0$,
so we assume $\mathcal{E}^\xi_0(u_0)<\infty$. Hence $u_0=\pm 1$ a.e.\ and $\text{Per}(C)<\infty$. As above, $u_0\in -1 + L^p(\mathbb{R}^d)$ and $u_0=\pm 1$ a.e.\ implies $|C|<\infty$. We will in the remainder of this proof allow our order symbols $o(\cdot)$, $O(\cdot)$ to depend on $\text{Per}(C)$ and $|C|$.

We first assume that $C$ is open, bounded, and with a $\mathcal{C}^2$ boundary. Letting $h(x)$ denote the signed  distance of the point $x\in \mathbb{R}^d$ to the boundary $\partial C$ (with the convention that $h(x)<0$ for $x \in C$), we set
\begin{align}\label{Gconstruction0}
u_\phi(x):=
\begin{cases}
w_\phi \left(\frac{h(x)}{\phi}\right)+\alpha _\phi, \quad & \mbox{for}\quad x \in \Omega_{\phi, L},\\
-1, \quad & \mbox{for}\quad x \in \mathbb{R}^d \setminus \Omega_{\phi, L},
\end{cases}
\end{align}
where $w_\phi :=v_R$ as in \eqref{v_R} with $R=\phi^{-1/2}$.  As usual, $\alpha _\phi$ is a constant  chosen so that $u_\phi$ satisfies the mean constraint $\dashint_{\Omega_{\phi, L}}  u_\phi \,dx=-1+\phi$.
We begin with an estimate of $\alpha_\phi$, which will be useful below. It follows from the mean constraint that
\begin{eqnarray}\label{Gconstruction1}
&&\alpha _\phi |\Omega _{\phi, L}| \notag \\
&=&(-1+\phi)|\Omega_{\phi, L}|-\int_{\Omega_{\phi, L}}w_\phi \left( \frac{h(x)}{\phi}\right)\,dx \notag \\
& \overset{\eqref{bul3}}=&\int_{\{h(x)<0\}}1-w_\phi\left(\frac{h(x)}{\phi}\right)\,dx-\int_{\{h(x)\geq 0\}}1+w_\phi\left(\frac{h(x)}{\phi}\right)\,dx \notag \\
&&\qquad +\xi^{d+1}-2|C|+o(1).
\end{eqnarray}
Using that $w_\phi(x)$ equals $\pm 1$ for $|x|>2\phi ^{-1/2}$, we obtain
\begin{align}\label{Gconstruction0.5}
&\left|\int_{\{h(x)<0\}}1-w_\phi\left(\frac{h(x)}{\phi}\right)\,dx-\int_{\{h(x)\geq 0\}}
1+w_\phi\left(\frac{h(x)}{\phi}\right)\,dx\right|\notag \\
& \leq\Big|\{|h(x)|<2\phi^{1/2}\}\Big|\lesssim \phi^{1/2}\text{Per}(C),
\end{align}
where in the second estimate we have used the coarea formula and
\begin{align}\label{Gconstruction1.5}
\text{Per}(\{h(x)> t\})\to \text{Per}(C) \quad \mbox{as}\quad t\to 0;
\end{align}
see for example \cite [lemma 2]{S}.

Substituting \eqref{Gconstruction0.5} into \eqref{Gconstruction1} and  recalling \eqref{bul3} yields
\begin{align}\label{Gconstruction2}
\alpha_\phi = \phi \left(1-\frac{2|C|}{\xi^{d+1}}\right)+o(\phi).
\end{align}

We now verify $u_\phi \to u_0$ in $-1 + L^p(\mathbb{R}^d)$. Indeed, we have
\begin{align}\label{Gconstruction2.1}
\|u_\phi-u_0\|^p_{L^p(\mathbb{R}^d)}&=\int_{\Omega_{\phi, L}}|u_\phi-u_0|^p\,dx \notag \\
&=\int_C\left|w_\phi \left(\frac{h(x)}{\phi}\right)+\alpha _\phi -1\right|^p\,dx \notag \\
&\qquad +\int_{\Omega_{\phi, L}\setminus C}\left|w_\phi \left(\frac{h(x)}{\phi}\right)+\alpha _\phi +1\right|^p\,dx \notag \\
&\leq |\alpha_\phi| ^p|\Omega_{\phi, L}|+2|\{|h(x)|<2\phi^{1/2}\}|.
\end{align}
Substituting \eqref{bul3}, \eqref {Gconstruction0.5}, and \eqref{Gconstruction2} into \eqref{Gconstruction2.1}, we deduce  that $\|u_\phi-u_0\|_{L^p(\mathbb{R}^d)}\to 0$ for $p>1$ as $(\phi,L)\to (0,\infty)$.

Next we estimate $\ephi (u_\phi)$. To begin, we write
\begin{align}\label{Gconstruction3}
&\ephi (u_\phi)\notag \\
&=\int_{\Omega_{\phi, L}}\frac{\phi}{2}\left|\nabla w_\phi\left(\frac{h(x)}{\phi}\right)\right|^2+\frac{1}{\phi}G\Big(w_\phi \Big(\frac{h(x)}{\phi}\Big)\Big)\,dx \notag \\
&\quad +\frac{\alpha_\phi}{\phi}   \int_{\Omega_{\phi, L}}     G'\Big(w_\phi \Big(\frac{h(x)}{\phi}\Big)\Big)\,dx\notag\\
&\quad+ \frac{1}{\phi}  \int_{\Omega_{\phi, L}}    \frac{\alpha_\phi^2}{2} G''\Big(w_\phi \Big(\frac{h(x)}{\phi}\Big)\Big) -G(-1+\phi)\,dx \notag \\
&\quad +\frac{\alpha_\phi^3}{3!\phi} \int_{\Omega_{\phi, L}}  G'''\Big(w_\phi \Big(\frac{h(x)}{\phi}\Big)\Big)\,dx+\frac{\alpha_\phi^4}{4!\phi} \int_{\Omega_{\phi, L}}  G^{(4)}\Big(w_\phi \Big(\frac{h(x)}{\phi}\Big)\Big)\,dx \notag \\
&=:I_0+I_1+I_2+I_3+I_4.
\end{align}

We now estimate each of the terms on the right-hand side. Recalling $|\nabla h(x)|=1$ and the fact that $w_\phi=\pm 1$ outside the interval $[-2\phi^{-1/2},2\phi^{1/2}]$, we use the coarea formula to obtain
\begin{align}\label{Gconstruction4}
I_0&=\frac{1}{\phi}\int_{\Omega_{\phi, L}}\left[\frac{1}{2}\Big(w_\phi'\Big(\frac{h(x)}{\phi}\Big)\Big)^2+G\Big(w_\phi \Big(\frac{h(x)}{\phi}\Big)\Big)\right]|\nabla h(x)|\,dx \notag \\
&=   \int_{-2\phi^{-1/2}}^{2\phi^{-1/2}}  \Big(\frac{1}{2}(w_\phi'(s))^2+G(w_\phi(s))\Big)\text{Per}(\{h>\phi s\})\,ds \notag \\
&=c_0\text{Per}(C)+o(1),
\end{align}
where we have applied \eqref{Gconstruction1.5}.
For $I_1$, we observe
\begin{align}\label{Gconstruction5}
|I_1|\leq \frac{|\alpha_\phi|}{\phi}|\{|h(x)|\leq 2\phi^{1/2}\}|=o(1),
\end{align}
where we took into account \eqref{Gconstruction0.5}, \eqref{Gconstruction2}, and $G'(\pm 1)=0$.

For the $I_2$ term, we use \eqref{Gconstruction0.5}, \eqref{Gconstruction2}, \eqref{bul3}, and $G''(\pm1)=2$  to deduce
\begin{align}
  \frac{\alpha_\phi^2}{2\phi} \int_{\Omega_{\phi, L}}G''\left(w_\phi \left(\frac{h(x)}{\phi}\right)\right)\,dx&=\frac{\alpha_\phi^2}{\phi}|\Omega_{\phi, L}|(1+o(1)) \notag \\
       &=\xi^{d+1}-4|C|+\frac{4|C|^2}{\xi^{d+1}}+o(1),\notag
\end{align}
so that, recalling \eqref{bul2}, we obtain
\begin{align}\label{Gconstruction6.6}
I_2=-4|C|+\frac{4|C|^2}{\xi^{d+1}}+o(1).
\end{align}
Regarding the remaining two terms, we easily find $I_3=o(1)$ and $I_4=o(1)$ which, combined with \eqref{Gconstruction3}, \eqref{Gconstruction4}, \eqref{Gconstruction5}, and \eqref{Gconstruction6.6}  yields
\begin{align*}
\lim_{\phi,L}\ephi (u_\phi)=c_0\text{Per}(C)-4|C|+4\xi^{-(d+1)}|C|^2=\mathcal{E}^\xi_0(u_0).
\end{align*}

Finally, we consider the case where $C$ is not open, bounded, and with $\mathcal{C}^2$ boundary. We find it convenient in this part to index our sequence with $j\in\mathbb{N}$, so that $\phi_j \to 0$ and $\phi_jL_j^{d/(d+1)}\to \xi$, and we study $\ephij$.  By an approximation theorem (cf. \cite [ Remark 13.12]{Ma}), for every $j\geq 1$ there exists an open and bounded set $C_j\subset \mathbb{R}^d$ with $\mathcal{C}^2$ boundary such that
\begin{align}\label{Gconstruction7}
|C_j\triangle C|\leq \frac{1}{j}\quad \mbox{and} \quad |\text{Per}(C_j)-\text{Per}(C)|\leq \frac{1}{j}.
\end{align}
Letting $h_j$ denote the signed distance from the boundary of $C_j$, we  choose $0<\phi_j < \frac{1}{j}$ small enough so that
$\{x\in \mathbb{R}^d: h_j(x)\leq 2\phi_j^{1/2}\}\subset \Omega_{\phi_j,L_j}$ and
\begin{align}\label{Gconstruction7.5}
|\text{Per}(\{h_j>\phi_js\})-\text{Per}(C_j)|<\frac{1}{j},\quad \forall s\in (-2\phi_j^{-1/2},2\phi_j^{-1/2}).
\end{align}
As in \eqref{Gconstruction0}, we define
\begin{align}\label{Gconstruction8}
u_{\phi_j}(x):=
\begin{cases}
w_{\phi_j} \left(\frac{h_j(x)}{\phi_j}\right)+\alpha _{\phi_j} \quad & \mbox{for}\quad x \in \Omega_{\phi_j,L_j}\\
-1 \quad & \mbox{for}\quad x \in \mathbb{R}^d \setminus \Omega_{\phi_j,L_j},
\end{cases}
\end{align}
where $w_{\phi_j} :=v_R$ as in \eqref{v_R} with $R=\phi_j^{-1/2}$, and the constant $\alpha _{\phi_j}$ is such that $\dashint_{\Omega_{\phi_j,L_j}}  u_{\phi_j} \,dx=-1+\phi_j$. Using this constraint, the first part of \eqref{Gconstruction7}, and \eqref{bul3}, one observes
\begin{align}
\alpha_j|\Omega_{\phi_j,L_j}|&=\phi_j |\Omega_{\phi_j,L_j}|-2|C_j|+\int_{\{h_j(x)<0\}}1-w_{\phi_j}\Big(\frac{h_j(x)}{\phi_j}\Big)\,dx \notag \\
& \qquad -\int_{\{h_j(x)>0\}}1+w_{\phi_j}\Big(\frac{h_j(x)}{\phi_j}\Big)\,dx+o(1) \notag \\
&=\xi^{d+1}-2|C|+o(1),
\end{align}
from which it follows, with another application of \eqref{bul3}, that
\begin{align}\label{Gconstruction9}
\alpha_j = \phi _j \left(1-\frac{2|C|}{\xi^{d+1}}\right)+o(\phi_j).
\end{align}
We therefore have
\begin{align}\label{Gconstruction9.5}
\|u_{\phi_j}-u_0\|_{L^p(\mathbb{R}^d)}&=\|u_{\phi_j}-(-1+2\chi_C)\|_{L^p(\mathbb{R}^d)}\notag \\
&\leq \|u_{\phi_j}+1-2\chi_{C_j}\|_{L^p(\mathbb{R}^d)} +2\|\chi_C-\chi_{C_j}\|_{L^p(\mathbb{R}^d)}\notag \\
&=\Big\|w_{\phi_j}\Big(\frac{h_j(x)}{\phi_j}\Big)+\alpha_j+1-2\chi_{C_j}\Big\|_{L^p(\Omega_{\phi_j,L_j})}+ 2|C_j\triangle C|^{1/p} \notag \\
&\leq |\alpha_j||\Omega_{\phi_j,L_j}|^{1/p}+|\{|h_j(x)|\leq2\phi_j^{1/2}\}|^{1/p}+2|C_j\triangle C|^{1/p} \notag \\
& \leq |\alpha_j||\Omega_{\phi_j,L_j}|^{1/p} +\phi_j^{1/2p} \left( 4\,\text{Per}(C) + o(1)\right)^{1/p} +o(1),
\end{align}
where we have argued as in \eqref{Gconstruction0.5} and applied the second part of \eqref{Gconstruction7} and \eqref{Gconstruction7.5}.
Substituting \eqref{Gconstruction9} and \eqref{bul3} into
 \eqref{Gconstruction9.5} yields
\begin{align}
\lim _{j\uparrow \infty}\|u_{\phi_j} -u_0\|_{L^p(\mathbb{R}^d)}=0.
\end{align}

It remains to estimate the energy.
Decomposing $\ephij(u_{\phi_j})$ as in \eqref{Gconstruction3} and estimating as in the previous case, we obtain
\begin{align*}
I_0&=c_0\text{Per}(C_j)+o(1)\overset{\eqref{Gconstruction7}}=c_0\text{Per}(C)+o(1),\\
I_2&=-4|C_j|+\frac{4|C_j|^2}{\xi^{d+1}}+o(1)\overset{\eqref{Gconstruction7}}=-4|C|+\frac{4|C|^2}{\xi^{d+1}}+o(1),\\
I_1&=o(1),\quad I_3=o(1),\quad I_4=o(1),
\end{align*}
as $j\uparrow \infty$, from which \eqref{Gconstruction-1} follows.
\end{proof}
\appendix
\section {The lowest energy saddle point $u_s$}
The existence of the energy barrier $\delE$ that separates the uniform state $\bar{u}$ from states of lower energy suffices to establish the existence of a saddle point $u_s$ of the energy functional $E(u)$, such that $E(u_s)=E(\bar{u})+\delE$. Here we define a saddle point of a $\mathcal{C}^1$ functional $E$ on a reflexive Banach space $X$ to be a point $x\in X$, such that $E'(x)=0$, and such that any neighborhood of $x$ contains two points $y$ and $z$ for which $E(y)<E(x)<E(z)$. In other words, a saddle point is a critical point that is neither a local maximum nor a local minimum of $E$.

A minimal energy saddle point $u_s$ on the boundary of the domain of attraction of the uniform state $\bar{u}$ is sometimes referred to as a critical nucleus.
We will use the ``minimax" characterization of $\delE$ (cf. \eqref{ebd}) and the mountain pass theorem to establish the existence of such a saddle point on the torus in the off-critical and critical regimes.

All of the arguments in the appendix are standard and we include them only for completeness.
We begin with the following definition.
\begin{definition}[Palais-Smale compactness criterion]
A sequence $x_k\in X$ is called a Palais-Smale sequence if $\sup_{k\geq 1}|E(x_k)|<\infty$ and $\|E'_{x_k}\|_{X^*}\to 0$.
A functional $E\in \mathcal{C}^1(X)$ is said to satisfy the Palais-Smale condition (PS) if every Palais-Smale sequence
has a strongly convergent subsequence in $X$.
\end{definition}
It is convenient to shift the argument of $E$ by the mean and consider the functional $\hat{\mathcal{E}}$  on the vector space $X$ of $w\in H^1\cap L^4(\Omega)$ with $\int_\Omega w\,dx=0$  by $\hat{\mathcal{E}}(w):=E(w+\bar{u})-E(\bar{u})$. Given that
\begin{align}\label{CH energy}
E(u)=\omegaint \frac{1}{2}|\nabla u|^2+\frac{1}{4}(1-u^2)^2\,dx,
\end{align}
we have
\begin{align}\label{I}
\hat{\mathcal{E}}(w)=\omegaint \frac{1}{2}|\nabla w|^2+\frac{w^4}{4}+\bar{u}w^3+\frac{1}{2}(3\bar{u}^2-1)w^2\,dx.
\end{align}
We define the norm on $X$ as $\|w\|:=\|\nabla w\|_2+\|w\|_4$, where $\|\cdot\|_p$ stands for the usual $L^p$-norm $\|\cdot\|_{L^p(\Omega)}$. We begin by checking that $\hat{\mathcal{E}}$ is smooth and satisfies PS.
\blemma\label{PS}
The functional $\hat{\mathcal{E}}$ is of class $\mathcal{C}^1(X)$ and satisfies the Palais-Smale condition.
\elemma
\begin{proof}
It is easy to see that $\hat{\mathcal{E}}$ is continuously Fr\'{e}chet differentiable in $X$, with its Fr\'{e}chet derivative at a point $w\in X$ defined via
\begin{align}\label{I'}
\hat{\mathcal{E}}'_w(\psi) = \omegaint \nabla w \cdot \nabla \psi + \left(w^3+3\bar{u}w^2+(3\bar{u}^2-1)w\right)\psi \,dx,
\end{align}
for all $\psi \in X$.
In order to verify the PS property, consider a PS sequence $\{w_k\}_{k\geq 1}\subset X$. Then $\hat{\mathcal{E}}(w_k)$ is uniformly bounded and, by  the coercivity of $\hat{\mathcal{E}}$,  we obtain  $\sup_{k\geq 1}\|w_k\|<\infty$. By passing to a subsequence, if necessary, we may assume that $w_k\rightharpoonup w$ in  $H^1(\Omega)$ and $L^4(\Omega)$.
Moreover, the compact imbedding of $H^1(\Omega)$ in $L^2(\Omega)$ implies $w_k\to w$ in $L^2(\Omega)$, so by interpolation and the boundedness of $\{w_k\}$ in $L^4(\Omega)$ we also obtain $w_k\to w$ in $L^3(\Omega)$.

On the other hand, since $\{w_k\}$ is a PS sequence, we also have
\begin{align}\label{I' to 0}
\hat{\mathcal{E}}'_{w_k}(w)\to 0 \quad \mbox{and} \quad \hat{\mathcal{E}}'_{w_k}(w_k)\to 0,
\end{align}
which, in light of \eqref{I'}, can be written as
\begin{align}\label{I'2}
\omegaint \nabla w_k\cdot \nabla w+\left(w_k^3+3\bar{u}w_k^2+(3\bar{u}^2-1)w_k\right)w\,dx\rightarrow 0,
\end{align}
and
\begin{align}\label{I'3}
\omegaint |\nabla w_k|^2+w_k^4+3\bar{u}w_k^3+(3\bar{u}^2-1)w_k^2\,dx\rightarrow 0,
\end{align}
respectively. From weak convergence in $H^1$ and strong convergence in $L^3$ together with \eqref{I'2} and \eqref{I'3}, we deduce
\begin{align*}
\lim_{k\uparrow \infty}\omegaint |\nabla w_k|^2+w_k^4\,dx&=\lim_{k\uparrow \infty}\omegaint \nabla w_k\cdot \nabla w+w_k^3w\,dx\\
&=\omegaint |\nabla w|^2\,dx+ \lim_{k\uparrow \infty} \omegaint w_k^3w\,dx,
\end{align*}
so that in order to deduce strong convergence of $w_k$ in $X$, it suffices to show
\begin{align}
  \omegaint w_k^3w\,dx \to \omegaint w^4 \,dx.\label{toshow}
\end{align}
By density of $L^\infty (\Omega)$  in $L^4(\Omega)$, for any $\eps>0$ there exists $g \in L^\infty (\Omega)$ such that $\|w-g\|_4 <\eps$. We thus have
\begin{align}
&\Big|\omegaint (w^3-w_k^3)w\,dx\Big|\leq \Big|\omegaint (w^3-w_k^3)(w-g)\,dx\Big|+\Big|\omegaint (w^3-w_k^3)g\,dx\Big| \nonumber \\& \lesssim \|w-g\|_4 \Big(\omegaint w^4+w_k^4\,dx\Big)^{3/4}+\omegaint |w-w_k|(w^2+w_k^2)|g|\,dx \nonumber \\&\lesssim\big(\eps+\|g\|_\infty \|w_k-w\|_2\big)(1+||w_k||_{L^4}^4).
\end{align}
By the uniform bound in $L^4$ and strong convergence in $L^2$, we obtain
\begin{align*}
\limsup _{k\uparrow \infty}\Big|\omegaint (w^3-w_k^3)w\,dx\Big|\leq C\eps\qquad\text{for some }C<\infty.
\end{align*}
Sending $\eps\to 0$, we obtain \eqref{toshow} and strong convergence of $w_k$ in $X$.
\end{proof}
We now employ a mountain pass argument to prove the existence of a saddle point $u_s$ of $E$ with $E(u_s)=E(\bar{u})+\delE$. This is the content of corollary \ref{cor:sad}.
\begin{proof}[Proof of corollary \ref{cor:sad}]
We will first show the existence of a nonconstant critical point with energy $E(\bar{u})+\delE$. Note that this is equivalent to showing the existence of a nonconstant critical point $w_s$ of $\hat{\mathcal{E}}$  such that $\hat{\mathcal{E}}(w_s)=\delE$.
Clearly $\hat{\mathcal{E}}(0)=0$, and as was shown in sections 2 and 3  we have $\mathcal{A}\neq \emptyset$ and $\delE>0$, where
\begin{align*}
\mathcal{A}:=\Big\{\gamma \in C([0,1];X):\gamma(0)=0, \hat{\mathcal{E}}(\gamma(1))<0\Big\},
\end{align*}
and
\begin{align}\label{def barrier 2}
\delE:=\inf_{\gamma \in \mathcal{A}}\max_{t\in [0,1]} \hat{\mathcal{E}}(\gamma(t)).
\end{align}

We denote the set of critical points of $\hat{\mathcal{E}}$ with critical value $\delE$ by $K_{\delE}$, i.e.,
\begin{align*}
K_{\delE}:=\Big\{w\in X: \hat{\mathcal{E}}'(w)=0 \quad \mbox{and} \quad \hat{\mathcal{E}}(w)=\delE\Big\}.
\end{align*}
Suppose  that $K_{\delE} = \emptyset$. Since by lemma \ref{PS} the functional $\hat{\mathcal{E}}$ satisfies PS, we may apply the deformation lemma (see, for example, theorem A.4. in \cite{R}), which implies that for any $\bar{\eps}\in (0, \delE)$ there exists some $\eps \in (0,\bar{\eps})$ and a homeomorphism $h:X\to X$ such that, with
\begin{align*}
A_s:=\Big\{w\in X: \hat{\mathcal{E}}(w)\leq s\Big\},
\end{align*}
we have
\begin{align}\label{prop of h1}
h(A_{\delE+\eps})\subset A_{\delE-\eps},
\end{align}
as well as
\begin{align}\label{prop of h2}
h(w)=w,\,\,\,\,\forall w\in X\quad \mbox{with}\quad \hat{\mathcal{E}}(w)\notin [\delE-\bar{\eps},\delE+\bar{\eps}].
\end{align}
By the fact that $\mathcal{A}\neq \emptyset$ and the definition of $\delE$, there exists  a path $\gamma_1 \in \mathcal{A}$ such that
\begin{align*}
\max_{t \in [0,1]}\hat{\mathcal{E}}(\gamma_1 (t))<\delE +\eps.
\end{align*}
Since $\max\{\hat{\mathcal{E}}(0),\hat{\mathcal{E}}(\gamma_1(1))\}=0<\delE -\bar{\eps}$, it  follows by \eqref{prop of h1} and \eqref{prop of h2} that $h \circ \gamma_1 \in \mathcal{A}$ and
\begin{align*}
\max_{t\in [0,1]}\hat{\mathcal{E}}(h(\gamma_1(t)))\leq \delE -\eps.
\end{align*}
This, however, contradicts \eqref{def barrier 2}, and therefore $K_{\delE}\neq \emptyset$.

It is a direct consequence of PS that the nonempty set $K_{\delE}$ is compact, and since $X$ is infinite dimensional it can be shown (cf., for example, \cite{PS}) that, since $K_{\delE}$ cannot separate two points in its complement $X\setminus K_{\delE}$, it must contain a saddle point.
\end{proof}
\bigskip
\section*{Acknowledgements}
We would like to thank Felix Otto for suggesting the problem and for insight into the scaling bound, which was important for the sharp lower bounds in this paper. In addition, we would like to thank Stan Alama, Lia Bronsard, Eric Carlen, Vladimir Delengov, Bob Kohn, Mark Peletier, Peter Sternberg, Michael Struwe, and Alfred Wagner for interesting discussions on this and related topics.
\medskip
%*****************************************************

\end{document}